\documentclass[12pt,reqno]{amsart}
\usepackage{geometry}                
\usepackage{color}
\usepackage{graphicx}
\usepackage{subfigure}
\usepackage{amssymb}
\usepackage{amsmath}
\usepackage{multirow}
\usepackage{enumerate}
\usepackage{epstopdf}
\usepackage{mathrsfs}
\usepackage{epic,eepic}
\usepackage{mathrsfs}
\usepackage{amsthm}

\begin{document}
\vfuzz2pt 
\hfuzz2pt 
\newtheorem{thm}{Theorem}[]
\newtheorem{model}{Model}
\newtheorem{pro}{Problem}[]
\newtheorem{cor}[thm]{Corollary}
\newtheorem{lem}[]{Lemma}[section]
\newtheorem{prop}[]{Proposition}[section]
\theoremstyle{definition}
\newtheorem{defn}[thm]{Definition}
\theoremstyle{remark}
\newtheorem{rem}[]{Remark}[section]
\numberwithin{equation}{section}
\newtheorem{col}{Conclusion}

\baselineskip 16pt

\title[]{\bf Global non-isentropic rotational supersonic flows in a semi-infinite divergent duct}%
\author[]{Geng Lai}%
\address{}%
\email{}%


\dedicatory{
Department of Mathematics, Shanghai University,
Shanghai, 200444, P.R. China\\ \vskip 4pt laigeng@shu.edu.cn}%



\begin{abstract}
Supersonic flows for the two-dimensional (2D) steady full Euler system are studied.
We construct a global non-isentropic rotational supersonic flow in a
 semi-infinite divergent duct. The flow satisfies the slip condition on the walls of the duct, and the state of the flow is given at the inlet of the duct.
The solution is constructed by the method of characteristics.
The main difficulty for the global existence is that uniform a priori $C^1$ norm estimate of the solution is hard to obtain, especially when the solution tends to vacuum state.
We derive a group of characteristic decompositions for the 2D steady full Euler system. Using these decompositions, we obtain uniform a priori estimates for the derivatives of the solution. A sufficient condition for the appearance of vacuum is also given.
We show that if there is a vacuum then the vacuum is always adjacent to one of the walls, and the interface between gas and vacuum must be straight. The method used here may be also used to construct other 2D steady non-isentropic rotational supersonic flows.

\
\vskip 4pt
\noindent%
\vskip 4pt
\noindent%
\end{abstract}

\maketitle

\section{\bf Introduction }

We consider the 2D steady compressible Euler system:
\begin{equation}
\left\{
  \begin{array}{ll}
    (\rho u)_x+(\rho v)_y=0, \\[4pt]
  (\rho u^2+p)_x+(\rho uv)_y=0,  \\[4pt]
   (\rho uv)_x+(\rho v^2+p)_y=0,\\[4pt]
(\rho uE +up)_x+(\rho vE+vp)_y=0,
  \end{array}
\right.\label{PSEU}
\end{equation}
where $(u, v)$ is the velocity, $\rho$ is the density, $p$ is the pressure, $E=\frac{u^2+v^2}{2}+e$ is the specific total energy, and $e$ is the specific internal energy.
For polytropic gases, we have the equations of state
$$
p=s\rho^{\gamma}\quad \mbox{and}\quad e=\frac{p\tau}{\gamma-1},
$$
where $s$ is the specific entropy and $\gamma>1$ is the adiabatic constant.

For smooth flow, system (\ref{PSEU}) can be written as
\begin{equation}
\left\{
  \begin{array}{ll}
    (\rho u)_x+(\rho v)_y=0, \\[4pt]
  uu_x+vu_y+\tau p_{x}=0,  \\[4pt]
    uv_x+vv_y+\tau p_{y}=0,\\[4pt]
us_x+vs_y=0.
  \end{array}
\right.\label{PsEuler}
\end{equation}

The eigenvalues of  (\ref{PsEuler}) are determined
by
\begin{equation}
\Big(\lambda-\frac{v}{u}\Big)^2\big[(v-\lambda u)^{2}-c^{2}(1+\lambda^{2})\big]=0,\label{characteristice}
\end{equation}
which yields
\begin{equation}
\lambda=\lambda_{\pm}(u,v,c)=\frac{uv\pm
c\sqrt{u^{2}+v^{2}-c^{2}}}{u^{2}-c^{2}}\quad\mbox{and}\quad\lambda=\lambda_0=\frac{v}{u}.
\end{equation}
Here, $c=\sqrt{\gamma s \rho^{\gamma-1}}$ is the sound speed.
So, if and only if $u^{2}+v^{2}>c^{2}$  (supersonic) system (\ref{PsEuler}) is hyperbolic and has two families of wave characteristics
defined as the integral curves of
$$C_{\pm}:\quad \frac{{\rm d}y}{{\rm d}x}=\lambda_{\pm}.$$ The stream characteristics are defined as the integral curves $$C_{0}:\quad \frac{{\rm d}y}{{\rm d}x}=\frac{v}{u}.$$

\begin{figure}[htbp]
\begin{center}
\includegraphics[scale=0.34]{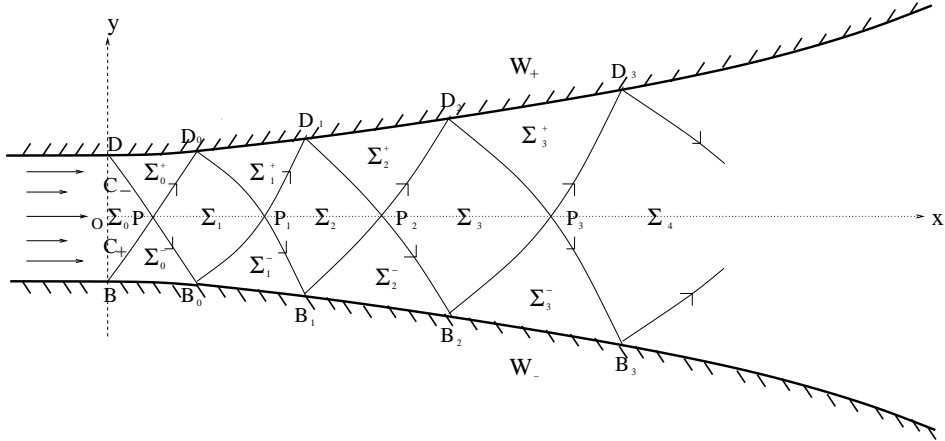}
\caption{ \footnotesize A piecewise smooth supersonic flow in a semi-infinite divergent duct.}
\label{Fig2}
\end{center}
\end{figure}

In this paper, we consider supersonic flows in a two-dimensional semi-infinite long divergent duct.  Assume that the duct denoted by $\Sigma$ is symmetric with respect to the x-axis and bounded by two walls $W_{+}$ and $W_{-}$ which are represented by
$$
W_{+}=\big\{(x, y) \mid y=f(x), ~~x\geq 0\big\}\quad \mbox{and}\quad  W_{-}=\big\{(x, y) \mid y=-f(x), ~~x\geq 0\big\},
$$
where $f(x)$ is assumed to satisfy:
\begin{equation}\label{72401}
 f(0)>0,\quad f'(0)= 0,\quad f''(x)> 0~~\mbox{as}~~x\geq 0,\quad
f_{\infty}'=\lim\limits_{x\rightarrow +\infty}f'(x)~~ \mbox{exists}.
\end{equation}
Then
$$
\Sigma=\{(x,y)\mid -f(x)<y<f(x), ~x>0\};
$$
see Figure \ref{Fig2}.
In order to study supersonic flows in such a duct,
we consider the following problem.
Assume that at the inlet of the duct, there is a supersonic incoming flow.
Then find a global supersonic flow in the duct $\Sigma$. 

When the incoming supersonic flow is a uniform flow, the global existence of continuous and piecewise smooth  solution in the duct was obtained by Chen and Qu in \cite{CQ2}.
\begin{thm} ({ Chen and Qu \cite{CQ2}})
Assume that the state of the incoming uniform supersonic flow is
$(u_0, 0, \rho_0, s_0)$ and that the duct is smooth and convex in the sense (\ref{72401}), where the constants $\rho_0>0$, $s_0>0$, and $u_0>c_0:=\sqrt{\gamma s_0\rho_0^{\gamma-1}}$.
Then there exists a global continuous and piecewise smooth solution in the duct.
Moreover, if $f_{\infty}'$ is greater than a constant determined by the Mach number $u_0/c_0$ and the adiabatic
exponent $\gamma$ of the incoming flow, then a vacuum will appear in the duct in finite
area. The vacuum region must be adjacent to the walls and the
boundary of the vacuum region is straight, which starts from and is tangential
to the curved part of the walls.
\end{thm}

A similar global existence was obtained by Wang and Xin by using potential-stream coordinates in \cite{WX1}.
When the incoming flow is sonic, the system
 becomes degenerate
at the inlet of the duct.
In a recent paper \cite{WX2}, Wang and Xin solved this degenerate hyperbolic problem and constructed a smooth transonic flow solution in a De Laval nozzle.
The supersonic flow constructed in \cite{CQ2} is isentropic and irrotational.
So, a natural question is to determine whether this result can be extended to the 2D steady full Euler system.
For this purpose, we consider system (\ref{PSEU}) with the boundary condition:
\begin{equation}\label{bd1}
\left\{
  \begin{array}{ll}
    (u, v, \rho, s)(0, y)=(u_{in}, v_{in}, \rho_{in}, s_{in})(y), & \hbox{$-f(0)\leq  y\leq f(0)$;} \\[6pt]
   (u, v)\cdot {\bf n_w}=0, & \hbox{$W_{+}\cup W_{-}$,}
  \end{array}
\right.
\end{equation}
where ${\bf n_w}$ denotes the normal vector of $W_{\pm}$, and $(u_{in}, v_{in}, \rho_{in}, s_{in})(y)$ satisfies
\begin{description}
  \item[(A1)] $(u_{in}, v_{in}, \rho_{in}, s_{in})(y)\in C^{1}[-f(0), f(0)]$;
  \item[(A2)] $v_{in}(y)=0$, $p_{in}(y)=s_{in}(y)\rho_{in}^{\gamma}(y)=\mbox{Const.}$ and
 $u_{in}(y)>c_{in}(y):=\sqrt{\gamma s_{in}(y)\rho_{in}^{\gamma-1}(y)}$ as $-f(0)\leq  y\leq f(0)$;
  \item[(A3)] $(u_{in}, v_{in}, \rho_{in}, s_{in})(y)=(u_{in}, v_{in}, \rho_{in}, s_{in})(-y)$ as $-f(0)\leq  y\leq f(0)$.
\end{description}
Actually, $(u, v, \rho, s)(x, y)=(u_{in}, v_{in}, \rho_{in}, s_{in})(y)$ is a laminar flow solution of system (\ref{PSEU}) under assumptions (A1)--(A3) .

Referring to Figure \ref{Fig25},
if the incoming flow is a uniform supersonic flow $(u_0, 0, \rho_0, s_0)$, then by the results of Courant and Friedrichs (\cite{CF}, Chap. IV.B) we know that there is a simple wave $\it S_{+}$ ($\it S_{-}$, resp.) with straight $C_{+}$ ($C_{-}$, resp.) characteristics
issuing from the lower wall $W_{-}$ (upper wall $W_{+}$, resp.).
These two simple waves start to interact with each other from a point $\bar{P}=\big(f(0)\sqrt{ u_0^2-c_0^2}/c_0
, 0\big)$.
Through $\bar{P}$ we draw a forward $C_{-}$ ($C_{+}$, resp.) characteristic curve $C_{-}^{\bar{P}}$ ($C_{+}^{\bar{P}}$, resp.) in
$\it S_{+}$ ($\it S_{-}$, resp.).
Then there are the following two cases:
\begin{description}
  \item[(i)] The characteristic curve $C_{-}^{\bar{P}}$ ($C_{+}^{\bar{P}}$, resp.) meets the lower wall $W_{-}$ (upper wall $W_{+}$, resp.) at a point $\bar{B}_0$ ($\bar{D}_0$, resp.), as indicated in Figure \ref{Fig25}(right).
  \item[(ii)] The characteristic curve $C_{+}^{\bar{P}}$ ($C_{-}^{\bar{P}}$, resp.) does not meet the lower wall $W_{-}$ (upper wall $W_{+}$, resp.).
\end{description}

\begin{figure}[htbp]
\begin{center}
\includegraphics[scale=0.35]{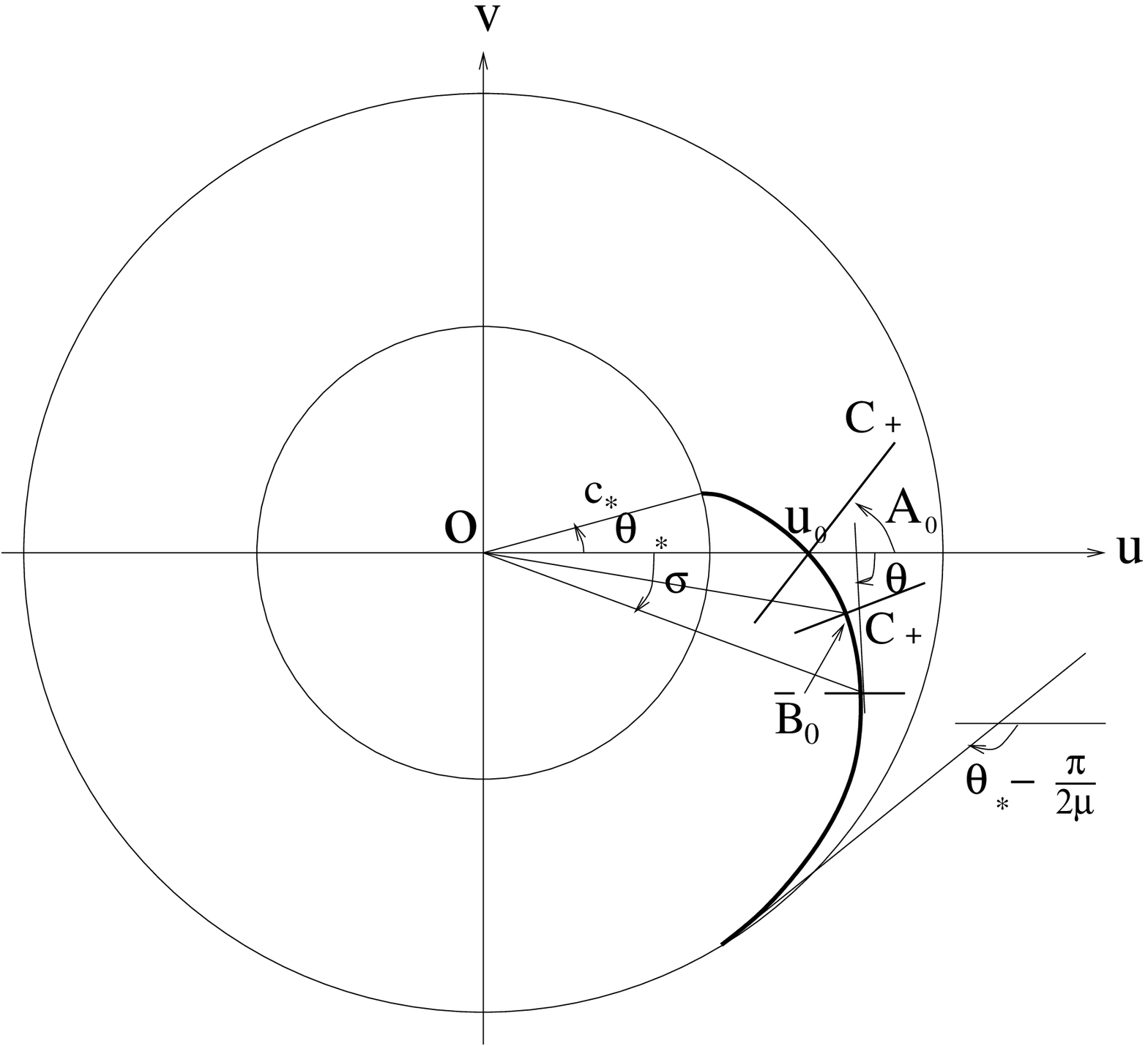} \qquad\qquad\includegraphics[scale=0.32]{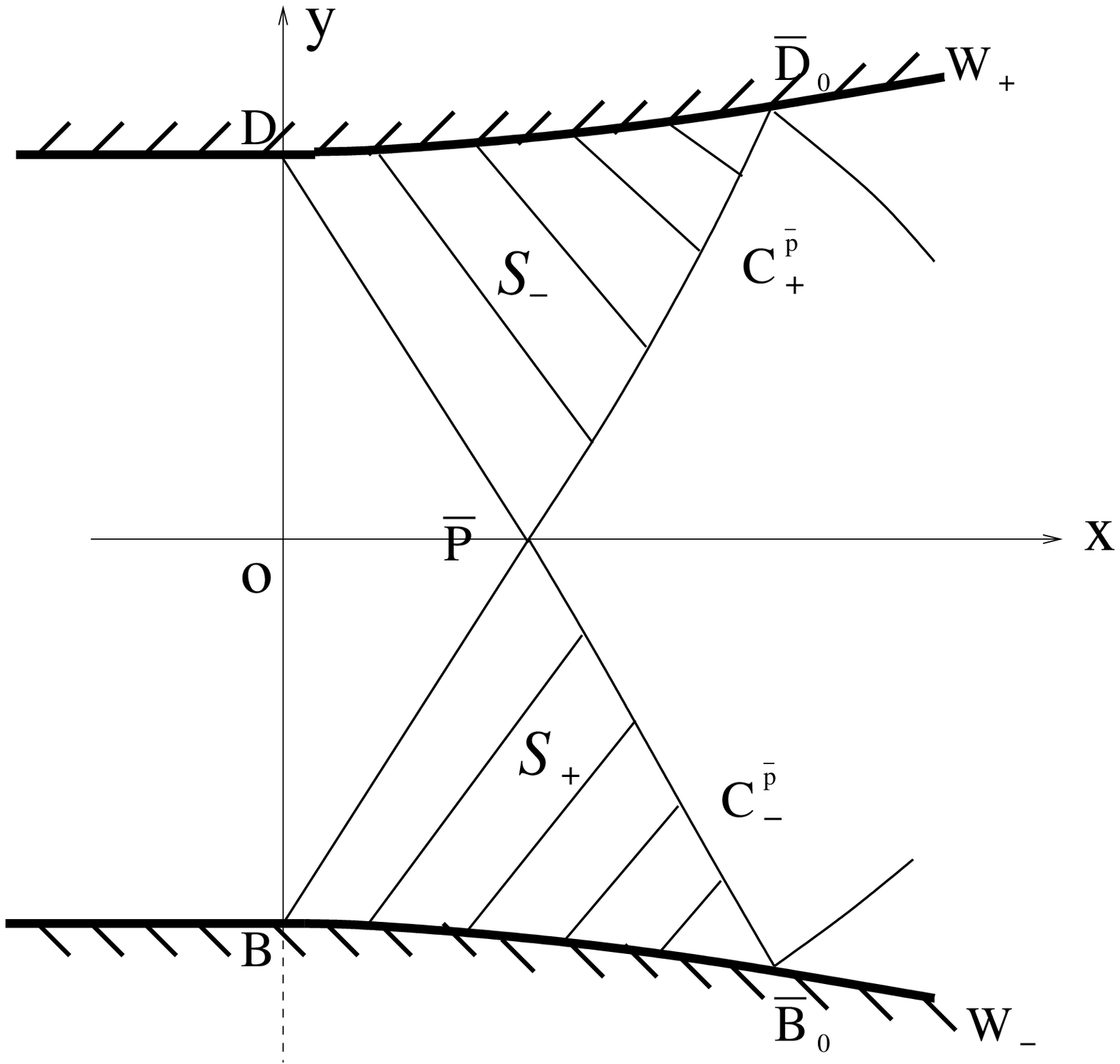}
\caption{ \footnotesize Simple waves adjacent to a constant state.}
\label{Fig25}
\end{center}
\end{figure}

In this paper, we show that for case (i), if the incoming flow is a small perturbation of the constant state $(u_0, 0, \rho_0, s_0)$ then the problem (\ref{PSEU}), (\ref{bd1}) admits a global continuous and piecewise smooth supersonic solution.
  Our main results can be stated as follows.
\begin{thm}\label{main}({\bf Main theorem})
Let $$\epsilon=\max\Big\{||u_{in}-u_0||_{_{C^1[-f(0), f(0)]}},~~
||\rho_{in}-\rho_0||_{_{C^1[-f(0), f(0)]}},~~||s_{in}-s_0||_{_{C^1[-f(0), f(0)]}}\Big\}.$$
Assume that $(u_{in}, v_{in}, \rho_{in}, s_{in})(y)$ satisfies (A1)--(A3) and the duct satisfies (\ref{72401}). For case (i), when $\epsilon$ is sufficiently small, the problem (\ref{PSEU}), (\ref{bd1})
admits a global continuous and piecewise smooth supersonic flow solution in the duct.
Moreover,
if $$\arctan f_{\infty}'>\frac{2}{\gamma-1}\cdot\frac{c_{in}\big(f(0)\big)}{ u_{in}\big(f(0)\big)}$$
then there are two vacuum regions adjacent to the walls and the interfaces between gas and vacuum are straight lines which start from and are tangential
to the walls.
\end{thm}

\begin{rem}
Actually, case (i) is vary easy to happen. In this case, we have the monotonicity conditions: $-\infty<\inf\limits_{\widehat{\bar{P}\bar{B}_0}}\bar{\partial}_{-}c<\sup\limits_{\widehat{\bar{P}\bar{B}_0}}\bar{\partial}_{-}c<0$ and $-\infty<\inf\limits_{\widehat{\bar{P}\bar{D}_0}}\bar{\partial}_{+}c
<\sup\limits_{\widehat{\bar{P}\bar{D}_0}}\bar{\partial}_{+}c<0$, where $\bar{\partial}_{\pm}$ are directional derivatives which will be defined later. These monotonicity conditions are crucial in constructing a global continuous and piecewise smooth solution to the problem (\ref{PSEU}), (\ref{bd1}) as $(u_{in}, v_{in}, \rho_{in}, s_{in})(y)$ is a small perturbation of the constant state $(u_0, 0, \rho_0, s_0)$.
\end{rem}

\begin{rem}
This theorem is also true if assumptions (A2) and (A3) are replaced by $v_{in}(f(0))=v_{in}(-f(0))=0$.
\end{rem}


We will use the method of characteristics,
so we need the concept of the direction of the wave characteristics.
The direction of the wave
characteristics is defined as the tangent direction that forms an
acute angle $A$ with the direction of the flow velocity
$(u, v)$. By simple computation, we see that the $C_+$
characteristic direction forms with the direction of the flow velocity
 the angle $A$ from $(u, v)$ to $C_{+}$ in
the counterclockwise direction, and the $C_-$ characteristic direction forms with the
 direction of the flow velocity the angle $A$ from $(u, v)$ to $C_{-}$
in the clockwise direction, as illustrated in Figure \ref{Fig1}.
 By computation, we have
\begin{equation}
c^{2}=q^{2}\sin^{2} A,\label{210cqo}
\end{equation}
in which $q^{2}=u^{2}+v^{2}$. The angle $A$ is called the
Mach angle.

\begin{figure}[htbp]
\begin{center}
\includegraphics[scale=0.53]{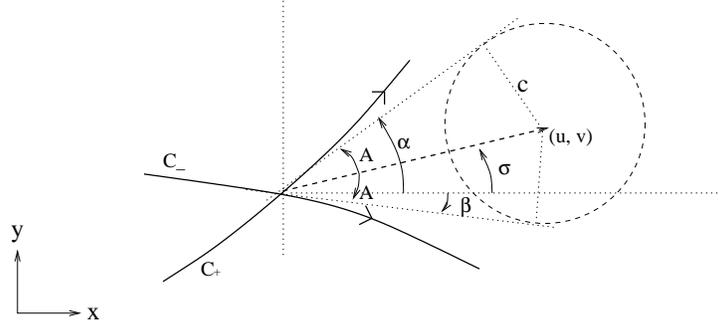}
\caption{ \footnotesize Characteristic curves, characteristic directions, and characteristic angles.}
\label{Fig1}
\end{center}
\end{figure}

Following \cite{CF} and \cite{Li3}, we use the concept of characteristic angle. The $C_{+}$ ($C_{-}$, resp.) characteristic angle is defined as the counterclockwise angle from the positive $x$-axis to the $C_{+}$ ($C_{-}$, resp.) characteristic direction.
We denote by $\alpha$ and
 $\beta$ the $C_{+}$ and $C_{-}$ characteristic angle,  respectively,
 where $0\leq \alpha-\beta\leq \pi$. Let $\sigma$ be the counterclockwise angle from the positive $x$-axis to the direction of the flow velocity. Obviously, we have
\begin{equation}
\alpha=\sigma+A,\quad \beta=\sigma-A,\quad\sigma=\frac{\alpha+\beta}{2},\quad A=\frac{\alpha-\beta}{2}, \label{tau}
\end{equation}
\begin{equation} \label{U}
u=q\cos\sigma, \quad v=q\sin\sigma,\quad u=c\frac{\cos\sigma}{\sin A},\quad \mbox{and}\quad v=c\frac{\sin\sigma}{\sin A}.
\end{equation}
By computation, we also have
\begin{equation}\label{4402}
q^{2}\cos\alpha\cos\beta=u^2-c^2\quad\mbox{and}\quad q^{2}\sin\alpha\sin\beta=v^2-c^2.
\end{equation}

Now, let us briefly describe the process of constructing a global continuous and piecewise smooth supersonic solution to the boundary value problem (\ref{PSEU}), (\ref{bd1}).

Referring to Figure \ref{Fig2},
 through point $B=(0, -f(0))$ draw a forward $C_{+}$ characteristic curve; through point $D=(0, f(0))$ draw a forward $C_{-}$ characteristic curve.
These two characteristic curves meet at some point $P$ on the $x-$axis. It can be seen by (A2) and (A3) that the flow in a region bounded by $\widehat{BP}$, $\widehat{DP}$, and $x=0$ is $(u, v, \rho, s)(x, y) =(u_{in}, v_{in}, \rho_{in}, s_{in})(y)$.
We solve a slip boundary problem for (\ref{PSEU}) in a region $\Sigma_{0}^{-}$ ($\Sigma_{0}^{+}$, resp.) bounded by $\widehat{BP}$ ($\widehat{DP}$, resp.), $W_{-}$ ($W_{+}$, resp.), and $\widehat{PB_0}$ ($\widehat{PD_0}$, resp.), where $\widehat{PB_0}$ ($\widehat{PD_0}$, resp.) is a $C_{-}$ characteristic curve which issues from $P$ and meets $W_{-}$ ($W_{+}$, resp.) at a point $B_0$ ($D_0$, resp.).
Meanwhile, in view of $f''>0$ and $\epsilon$ is sufficiently small we will get
$$
\sup\limits_{\widehat{PB_0}}\bar\partial_{-}c<0\quad\mbox{and}\quad \sup\limits_{\widehat{PD_0}}\bar\partial_{+}c<0.
$$
where the directional derivatives
\begin{equation}
\bar{\partial}_+=\cos\alpha\partial_{x}+\sin\alpha\partial_{y}\quad \mbox{and}\quad \bar{\partial}_{-}=\cos\beta\partial_{x}+\sin\beta\partial_{y}.
\end{equation}


We will then solve a Goursat problem for (\ref{PSEU}) in a region $\Sigma_1$
bounded by $\widehat{PD_0}$, $\widehat{PB_0}$, a forward $C_{-}$ characteristic curve $C_{-}^{D_0}$ issuing from $D_0$, and a forward $C_{+}$ characteristic curve $C_{+}^{B_0}$ issuing from $B_0$. There are two cases: one is that $C_{+}^{B_0}$ and $C_{-}^{D_0}$ intersect at some point $P_1$ as indicated in Figure \ref{Domain}(1); the other is that $C_{+}^{B_0}$ and $C_{-}^{D_0}$ do not intersect with each other as indicated in Figure \ref{Domain}(2).

\begin{figure}[htbp]
\begin{center}
\includegraphics[scale=0.296]{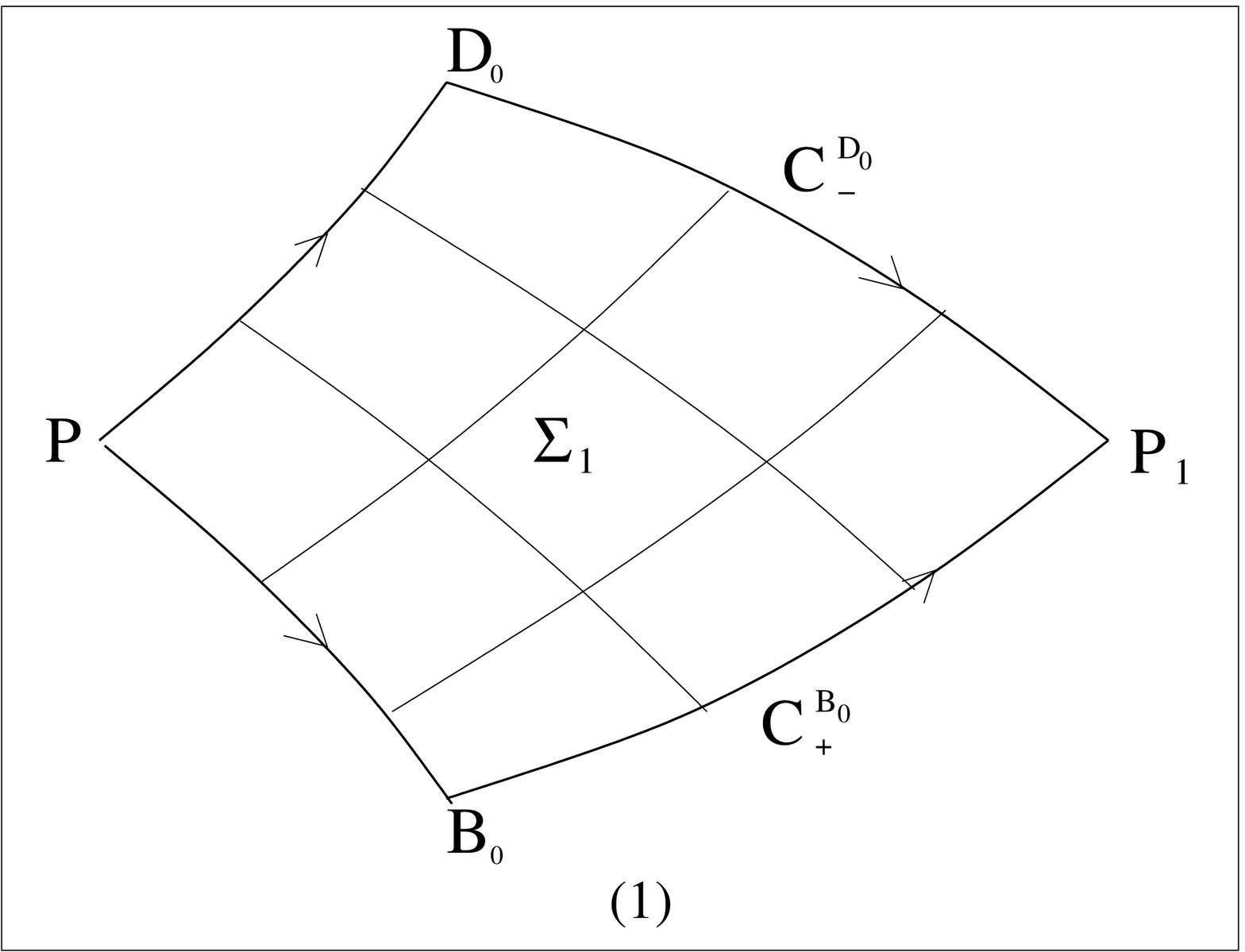}~~\includegraphics[scale=0.26]{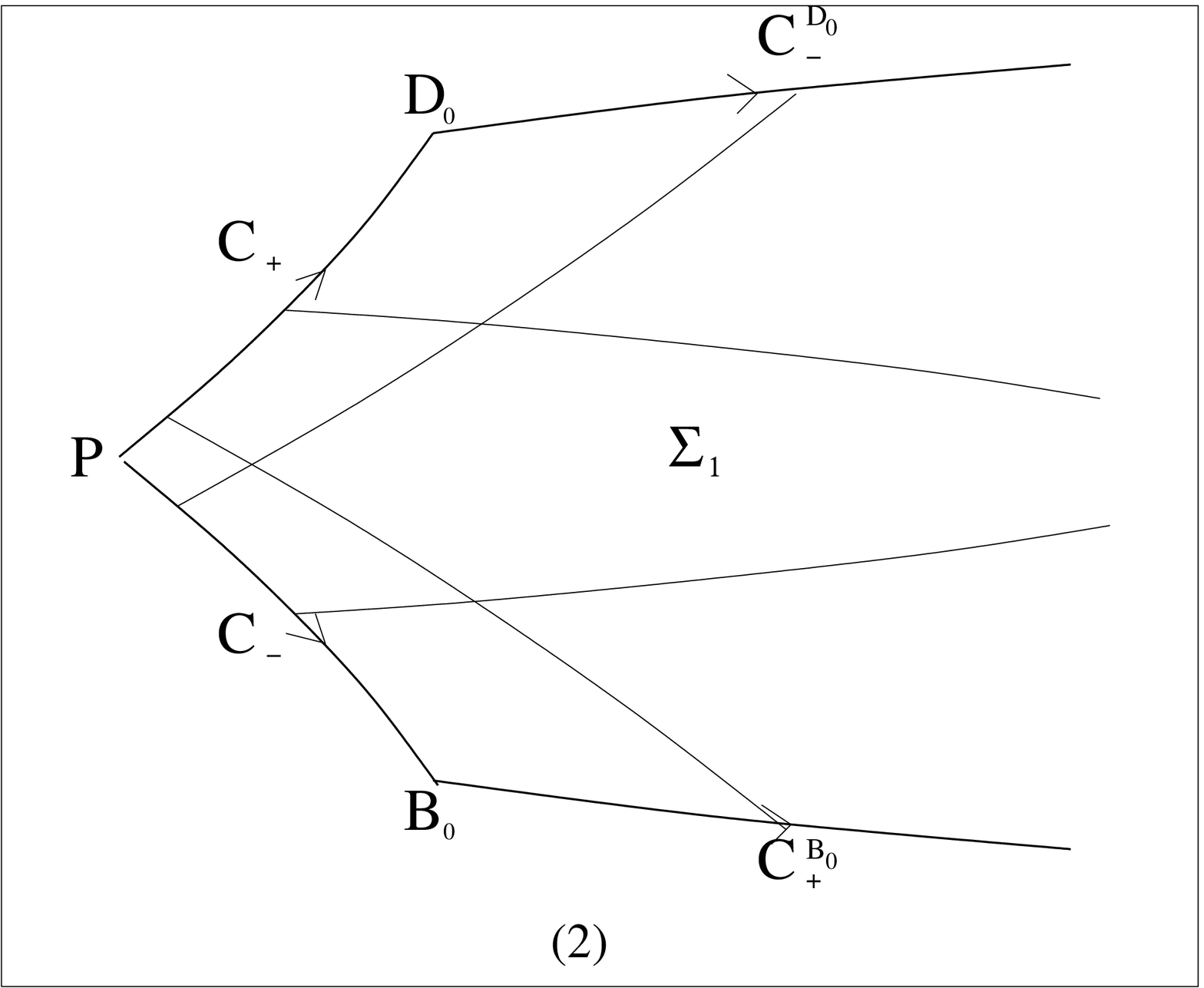}\\
\vskip 2pt
\includegraphics[scale=0.345]{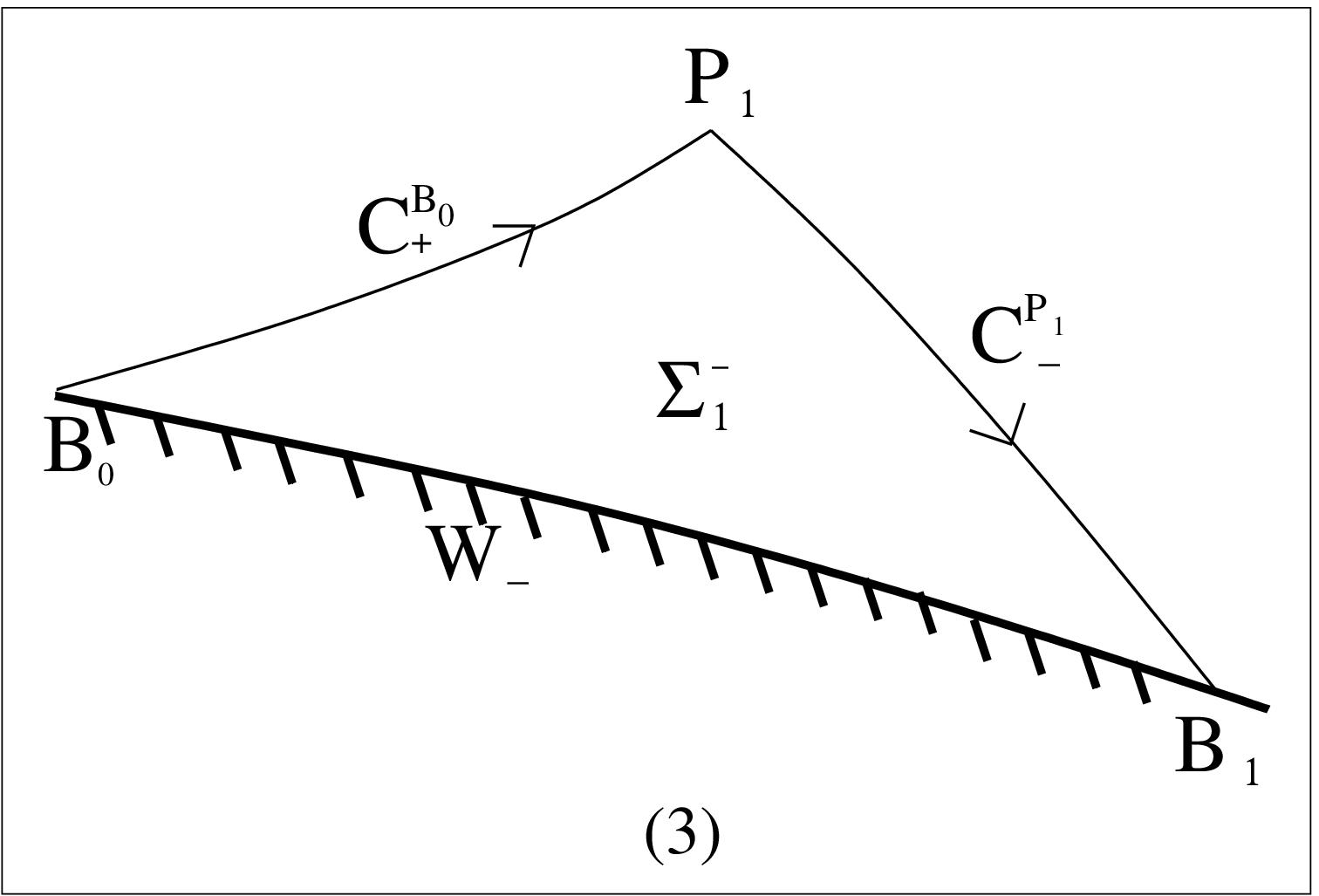}~ \includegraphics[scale=0.325]{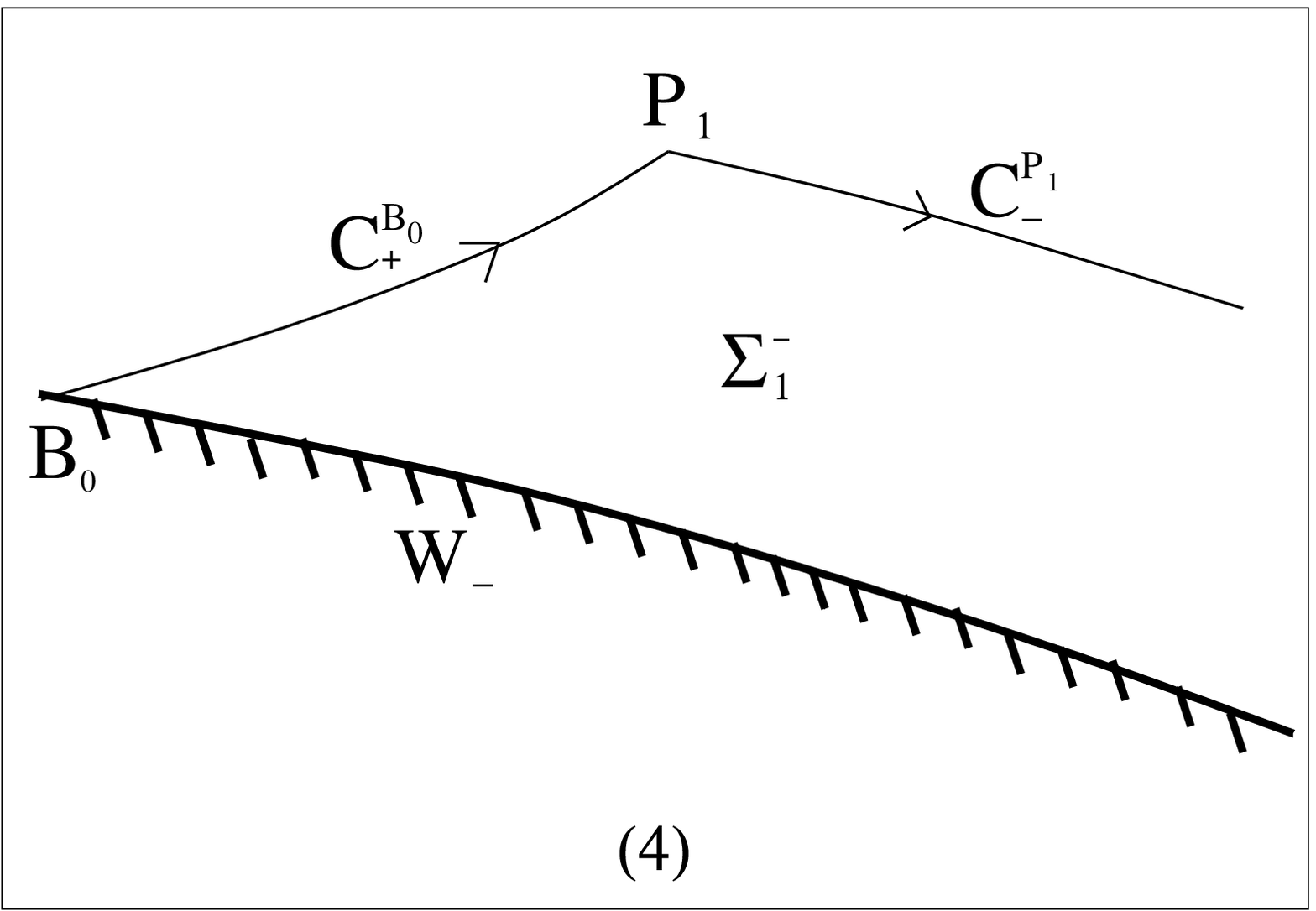}~\includegraphics[scale=0.36]{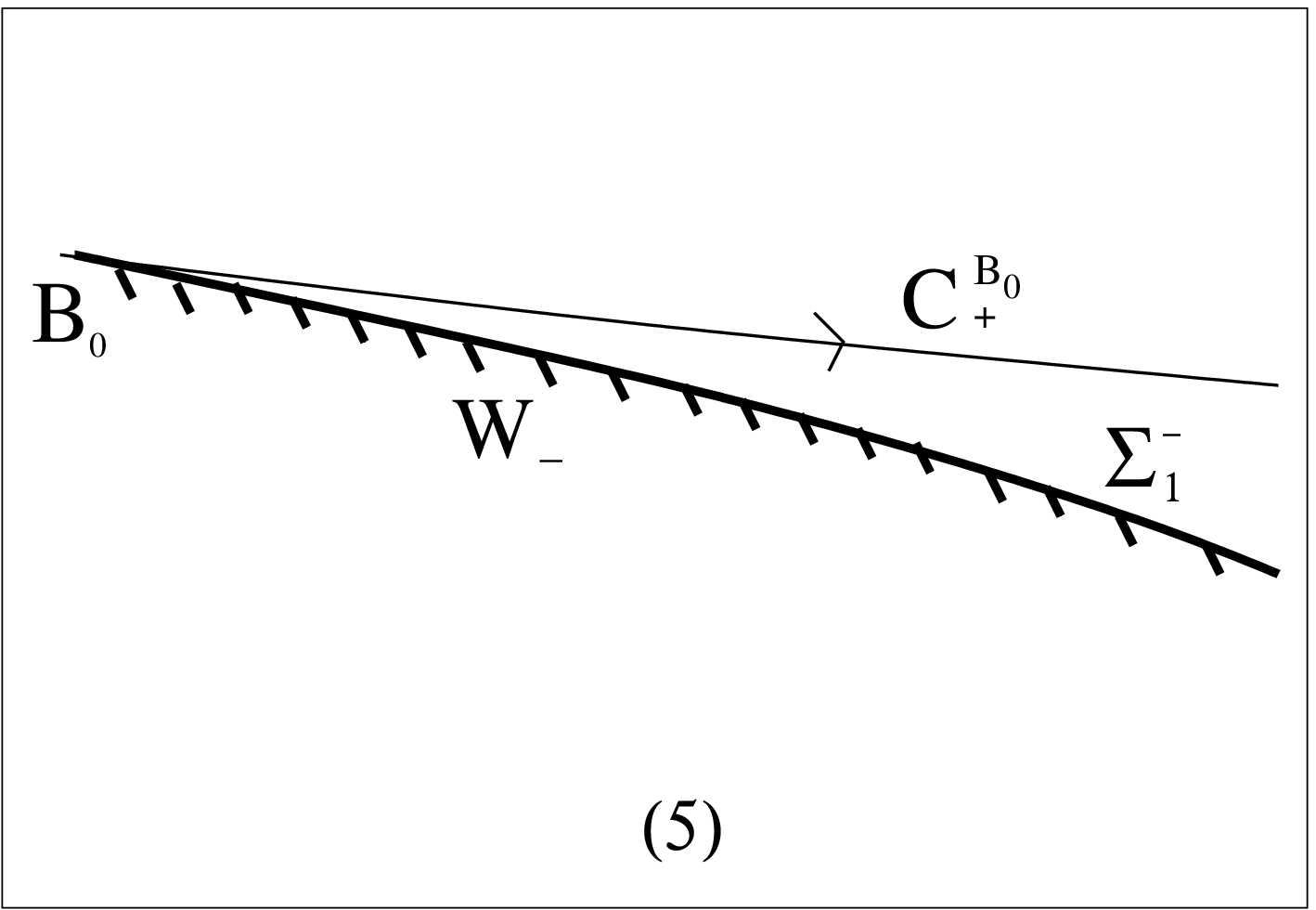}
\caption{ \footnotesize Domains $\Sigma_1$ and $\Sigma_1^{-}$.}
\label{Domain}
\end{center}
\end{figure}

Next, we solve a slip boundary problem for (\ref{PSEU}) in a region $\Sigma_1^{-}$ adjacent to $W_{-}$.  If $C_{+}^{B_0}$ and $C_{-}^{D_0}$ intersect at some point $P_1$, then there are two possibilities about $\Sigma_1^{-}$.
One is that the $C_{-}$ characteristic curve $C_{-}^{P_1}$ issuing from $P_1$ intersects with $W_{-}$ at a point $B_1$, and then $\Sigma_1^{-}$ is a bounded domain closed by $\widehat{B_0P_1}$, $\widehat{P_1B_1}$, and $W_{-}$; see Figure \ref{Domain}(3).
The other is that $C_{-}^{P_1}$ does not intersect with $W_{-}$, and then $\Sigma_1^{-}$ is an infinite region bounded by $\widehat{BP_1}$, $W_{-}$, and $C_{-}^{P_1}$; see Figure \ref{Domain}(4).
If $C_{+}^{B_0}$ and $C_{-}^{D_0}$ do not intersect with each other, then $\Sigma_{1}^{-}$ will be an infinite region between $C_{+}^{B_0}$ and $W_{-}$; see Figure \ref{Domain}(5).
By symmetry, one can obtain the flow in a region $\Sigma_1^{+}$  adjacent to $W_{+}$.

If $C_{+}^{B_0}$ and $C_{-}^{D_0}$ intersect at some point $P_1$ then we continue to solve a Goursat problem for (\ref{PSEU}) with $C_{-}^{P_1}$ and $C_{+}^{P_1}$ as the characteristic boundaries in a region $\Sigma_2$.
By repeatedly solving similar Goursat problems and slip boundary problems, one can get the solution
in regions $\Sigma_{0}$, $\Sigma_{0}^{+}$, $\Sigma_{0}^{-}$, $\Sigma_{1}$, $\Sigma_{1}^{+}$, $\Sigma_{1}^{-}$, $\Sigma_{2}$, $\Sigma_{2}^{+}$, $\Sigma_{2}^{-}$, $\cdot\cdot\cdot$, as illustrated in Figure \ref{Fig2}.
Finally, we shall show that one can obtain a global piecewise smooth supersonic flow solution in the domain $\Sigma$ after solving a finite number of Gourst problems and slip boundary problems.

One of the main difficulties to construct the global solution is that uniform a priori $C^1$ norm estimate of solution is hard to obtain, especially when the solution tends to vacuum state.
If the flow is isentropic and irrotational then system (\ref{PSEU}) can by using Riemann invariants be reduced to a $2\times 2$ system, and the existence of global classical solution to Goursat problem and slip boundary problem can be obtained  by monotonicity conditions of the boundary data; see Li \cite{LiT}. However, if the flow is non-isentropic and rotational then system (\ref{PSEU}) is a $4\times 4$ system and cannot be diagonalized.
Thus the methods in \cite{CQ2, LiT, WX1} do not work here.
In this paper,
we derive the following two important characteristic equations about the entropy and the vorticity $\omega=u_y-v_x$:
$$
\bar{\partial}_{0}\Big(\frac{\bar{\partial}_{\pm} s}{c^{\frac{\gamma+1}{\gamma-1}}}\Big)
~=~0\quad \mbox{and}\quad \bar{\partial}_{0}\Big(\frac{\omega}{\rho}\Big)
=-\Big(\frac{\bar{\partial}_{+}s}{c^{\frac{\gamma+1}{\gamma-1}}}\Big)\frac{(s\gamma)^{\frac{1}{\gamma-1}} }{\gamma(\gamma-1)s}\bar{\partial}_{0}c^2,
$$
where
\begin{equation}\label{72801}
\bar{\partial}_0=\cos\sigma\partial_{x}+\sin\sigma\partial_{y}.
\end{equation}
Using these equations, we can control the bounds of $\frac{\bar{\partial}_{\pm} s}{c^{\frac{\gamma+1}{\gamma-1}}}$ and $\frac{\omega}{\rho}$. (Remark: we shall show that although the solution is piecewise smooth in the duct, $\bar{\partial}_{\pm} s$ and $\omega$ are actually continuous in the duct.)
We also derive a group of characteristic decompositions (that can be actually seen as a system of ``ordinary differential equations") about $R_{+}$, $R_{-}$, $\frac{\bar{\partial}_{+} s}{c^{\frac{\gamma+1}{\gamma-1}}}$, and $\frac{\omega}{\rho}$, where
\begin{equation}\label{1973101}
R_{\pm}:=\bar{\partial}_{\pm}c-\frac{j\bar{\partial}_{\pm}s}{\kappa}=
\bar{\partial}_{\pm}c-\frac{1}{2\gamma s}\Big(\frac{\bar{\partial}_{\pm} s}{c^{\frac{\gamma+1}{\gamma-1}}}\Big)c^{\frac{2\gamma}{\gamma-1}}
,
\end{equation}
$$
\kappa=\frac{2}{\gamma-1},\quad\mbox{and} \quad j=\frac{c}{\gamma (\gamma-1)s};
$$
see (\ref{cd8}) and (\ref{cd10}).
We first use the characteristic decompositions (\ref{cd10}) to prove that $c^{-\frac{2\gamma}{\gamma-1}}R_{+}$ and $c^{-\frac{2\gamma}{\gamma-1}}R_{-}$
have an identical uniform negative upper bound as $\epsilon$ is sufficiently small,
and then use this negative upper bound and the characteristic decompositions (\ref{cd8}) to prove that $R_{+}$ ($R_{-}$, resp.) is monotone increasing along $C_{-}$ ($C_{+}$, resp.) characteristic curves in the sense of characteristic direction, and consequently get a uniform negative lower bound of $R_{+}$ and $R_{-}$.
The estimations of the derivatives of $u$ and $v$ can then be obtained by  (\ref{192303})--(\ref{82505}).

In order to prove that the global solution can be constructed after solving a finite number of Gourst problems and slip boundary problems, the method of hodograph transformation in \cite{CQ2} does not work here, since we do not have Riemann invariants for the 2D steady full Euler system. In this paper, we use characteristic angles $\alpha$ and $\beta$.
By the uniform negative upper bound of $c^{-\frac{2\gamma}{\gamma-1}}\bar{\partial}_{\pm}c$ and the characteristic equations about $\alpha$ and $\beta$ (see (\ref{192307}) and (\ref{192309})), we will see that the wave characteristic curves are convex as $c$ is sufficiently small.
Using the convexity of the wave characteristic curves, we will
prove that $\Sigma$ can be covered by a finite number of determinate regions of those Goursat problems and slip boundary value problems; we will also prove that if there is a vacuum then the vacuum is always adjacent to one of the walls and the interface between gas and vacuum must straight.

In our previous studies \cite{Lai5, Lai6}, we constructed several non-isentropic rotational supersonic flow solutions for the 2D (pseudo-)steady Euler equations. However, these solutions were constructed in bounded regions and under the assumption that $c$ has a non-zero lower bound. In the present paper, we overcome the difficulty caused by vacuum.

The rest of the paper is organized as follows. Section 2 is concerned with characteristic equations of the 2D steady full Euler system. In section 2.1, we derive a group of first order characteristic equations
for the variables $\alpha$, $\beta$, $c$, and $s$.
In section 2.2, we derive a group characteristic
equations about $R_{+}$, $R_{-}$, $c^{-\frac{\gamma+1}{\gamma-1}}\bar{\partial}_{+} s$, and $\frac{\omega}{\rho}$. Section 3 is devoted to construct a global supersonic flow solution to the boundary value problem (\ref{PSEU}), (\ref{bd1}).

\section{\bf Characteristic decompositions of the 2D steady Euler equations}

\subsection{Characteristic equations}
Set
\begin{equation}
\widehat{E}=\frac{q^{2}}{2}+\frac{c^2}{\gamma-1}.\label{Ber}
\end{equation}
Then by the last three equations of (\ref{PsEuler}) we have
\begin{equation}\label{4401}
u\widehat{E}_{x}+v\widehat{E}_{y}=0\quad \mbox{and}\quad \bar{\partial}_{0}\widehat{E}=0.
\end{equation}

From (\ref{Ber}) we also have
\begin{equation}\label{72301}
\rho_{x}=\frac{1}{c^2\tau}\Big(\widehat{E}_{x}-uu_{x}-vv_{x}-\frac{\gamma\rho^{\gamma-1}s_x}{\gamma-1}\Big)
\quad\mbox{and}\quad
\rho_{y}=\frac{1}{c^2\tau}\Big(\widehat{E}_{y}-uu_{y}-vv_{y}-\frac{\gamma\rho^{\gamma-1}s_y}{\gamma-1}\Big).
\end{equation}
Inserting this into the first equation of (\ref{PsEuler}) and using (\ref{4401}) and the forth equation of (\ref{PsEuler}), we get
\begin{equation}
  (c^{2}-u^{2})u_{x}-uv(u_{y}+v_{x})+(c^{2}-v^{2})v_{y}=0.\label{mass}
\end{equation}

Multiplying system
$$
\begin{array}{rcl}
\left(
 \begin{array}{cc}
c^{2}-u^{2} & -uv \\
  0 & -1\\
  \end{array}
  \right)\left(
           \begin{array}{c}
             u \\
             v \\
           \end{array}
         \right)_{x}+\left(
                         \begin{array}{ccccc}
                          -uv &  c^{2}-v^{2}\\
                           1 & 0 \\
                         \end{array}
                       \right)\left(
                                \begin{array}{c}
                                  u \\
                                  v \\
                                \end{array}
                              \right)_{y}=\left(
                                               \begin{array}{c}
                                                 0 \\
                                                 \omega \\
                                               \end{array}
                                             \right)
                                             \label{matrix1}
                                             \end{array}
$$ on the left
by
$(1,\mp c\sqrt{u^{2}+v^{2}-c^{2}})$ and using (\ref{4402}), we get
\begin{equation}
\left\{
  \begin{array}{ll}
  \displaystyle \bar{\partial}_{+}u+\lambda_{-}\bar{\partial}_{+}v
 =\frac{\omega\sin A\cos A}{\cos\beta},  \\[10pt]
    \displaystyle  \bar{\partial}_{-}u+\lambda_{+}\bar{\partial}_{-}v=-\frac{\omega\sin A\cos A}{\cos\alpha}.
  \end{array}
\right.\label{form}
\end{equation}

From (\ref{tau}) and (\ref{U}) we have
\begin{equation}\label{72701}
\bar{\partial}_{\pm}u=\frac{\cos\sigma}{\sin A}\bar{\partial}_{\pm}c+\frac{c\cos\alpha\bar{\partial}_{\pm}\beta
-c\cos\beta\bar{\partial}_{\pm}\alpha}{2\sin^{2}A}
\end{equation}
and
\begin{equation}\label{72702}
\bar{\partial}_{\pm}v=\frac{\sin\sigma}{\sin A}\bar{\partial}_{\pm}c
+\frac{c\sin\alpha\bar{\partial}_{\pm}\beta
-c\sin\beta\bar{\partial}_{\pm}\alpha}{2\sin^{2}A}.
\end{equation}

Inserting (\ref{72701}) and (\ref{72702}) into (\ref{form}), we obtain
\begin{equation}\label{3}
\bar{\partial}_{+}c=\frac{c}{\sin2A}
(\bar{\partial}_{+}\alpha-\cos2A\bar{\partial}_{+}\beta)+\omega\sin^{2}A
\end{equation}
and
\begin{equation}
\bar{\partial}_{-}c=\frac{c}{\sin2A}
(\cos2A\bar{\partial}_{-}\alpha-\bar{\partial}_{-}\beta)-\omega\sin^{2}A.\label{4}
\end{equation}

From the second and the third equations of (\ref{PsEuler}) we have
$$
\widehat{E}_x=-v\omega+\frac{\rho^{\gamma-1}}{\gamma-1}s_x\quad \mbox{and} \quad \widehat{E}_y=u\omega+\frac{\rho^{\gamma-1}}{\gamma-1}s_y.
$$
Hence, we have
\begin{equation}\label{726021}
u\bar{\partial}_{+}u+v\bar{\partial}_{+}v+\kappa c\bar{\partial}_{+}c=-v\omega\cos\alpha +u\omega\sin\alpha+\frac{\rho^{\gamma-1}}{\gamma-1}\bar{\partial}_{+}s
\end{equation}
and
\begin{equation}\label{41102}
u\bar{\partial}_{-}u+v\bar{\partial}_{-}v+\kappa c\bar{\partial}_{-}c=-v\omega\cos\beta +u\omega\sin\beta+\frac{\rho^{\gamma-1}}{\gamma-1}\bar{\partial}_{-}s.
\end{equation}

Inserting (\ref{72701})--(\ref{72702}) into (\ref{726021}) and (\ref{41102}), we get
\begin{equation}
\left(\frac{1}{\sin^{2}A}+\kappa\right)\bar{\partial}_{+}c=
\frac{c\cos A}{2\sin^{3}A}(\bar{\partial}_{+}\alpha-\bar{\partial}_{+}\beta)
+\omega+j\bar{\partial}_{+}s\label{bB1}
\end{equation}
and
\begin{equation}
\left(\frac{1}{\sin^{2}A}+\kappa \right)\bar{\partial}_{-}c=
\frac{c\cos A}{2\sin^{3}A}(\bar{\partial}_{-}\alpha-\bar{\partial}_{-}\beta)
-\omega+\bar{\partial}_{-}s,\label{bB2}
\end{equation}
respectively.

Inserting (\ref{3}) into (\ref{bB1}), we obtain
\begin{equation}\label{6}
c\bar{\partial}_{+}\alpha=\Omega\cos^{2}Ac\bar{\partial}_{+}\beta-\Big(\frac{\kappa\sin2A\sin^{2} A}{1+\kappa}\Big)\omega+\frac{j\sin 2A}{1+\kappa}\bar{\partial}_{+}s,
\end{equation}
where
$$
\Omega=\frac{\kappa-1}{\kappa+1}-\tan^2 A.
$$
Inserting (\ref{4}) into (\ref{bB2}), we obtain
\begin{equation}\label{5}
c\bar{\partial}_{-}\beta=\Omega\cos^{2}Ac\bar{\partial}_{-}\alpha-\Big(\frac{\kappa\sin2A\sin^{2} A}{1+\kappa}\Big)\omega-\frac{j\sin 2A}{1+\kappa}\bar{\partial}_{-}s.
\end{equation}

Combining with (\ref{3}) and (\ref{6}), we have
\begin{equation}\label{192306}
c\bar{\partial}_{+}\beta=-(1+\kappa)\tan A\bar{\partial}_{+}c+\omega\sin^2A\tan A+j\tan A \bar{\partial}_{+}s
\end{equation}
and
\begin{equation}\label{192307}
c\bar{\partial}_{+}\alpha=-\left(\frac{1+\kappa}{2}\right)\Omega\sin2 A\bar{\partial}_{+}c-\omega\sin^2A\tan A
+j\tan A \cos 2A\bar{\partial}_{+}s.
\end{equation}

Combining with (\ref{4}) and (\ref{5}), we have
\begin{equation}\label{192308}
c\bar{\partial}_{-}\alpha=(1+\kappa)\tan A\bar{\partial}_{-}c+\omega\sin^2A\tan A-j\tan A \bar{\partial}_{-}s
\end{equation}
and
\begin{equation}\label{192309}
c\bar{\partial}_{-}\beta=\left(\frac{1+\kappa}{2}\right) \Omega\sin2 A\bar{\partial}_{-}c-\omega\sin^2A\tan A
-j\tan A\cos2A \bar{\partial}_{-}s.
\end{equation}

Inserting (\ref{192306})--(\ref{192309}) into (\ref{72701}) and (\ref{72702}), we get
\begin{equation}\label{192303}
\bar{\partial}_{+}u=\kappa\sin\beta\bar{\partial}_{+}c+\omega\cos \sigma\sin A-j\sin \beta \bar{\partial}_{+}s,
\end{equation}
\begin{equation}\label{192304}
\bar{\partial}_{-}u=-\kappa\sin\alpha\bar{\partial}_{-}c-\omega\cos \sigma\sin A+j\sin \alpha \bar{\partial}_{-}s,
\end{equation}
\begin{equation}
\bar{\partial}_{+}v=-\kappa\cos\beta\bar{\partial}_{+}c+\omega\sin\sigma\sin A+j\cos \beta \bar{\partial}_{+}s,
\end{equation}
\begin{equation}\label{82505}
\bar{\partial}_{-}v=\kappa\cos\alpha\bar{\partial}_{-}c-\omega\sin\sigma\sin A-j\cos \alpha \bar{\partial}_{-}s.
\end{equation}

From (\ref{72801}) we have
\begin{equation}\label{72802}
\partial_{x}=-\frac{\sin\beta\bar{\partial}_{+}-\sin\alpha\bar{\partial}_{-}}{\sin2A},\quad
\partial_{y}=\frac{\cos\beta\bar{\partial}_{+}-\cos\alpha\bar{\partial}_{-}}{\sin2A},
\end{equation}
and
\begin{equation}\label{32403}
\bar{\partial}_{+}+\bar{\partial}_{-}=2\cos A \bar{\partial}_{0}.
\end{equation}
Thus, by $ \bar{\partial}_{0}s=0$ we have
\begin{equation}\label{32402}
c(\bar{\partial}_{+}j+\bar{\partial}_{-}j)=j(\bar{\partial}_{+}c+\bar{\partial}_{-}c)
\end{equation}
and
\begin{equation}\label{42001}
\bar{\partial}_{+}s~=~-\bar{\partial}_{-}s.
\end{equation}

\subsection{Characteristic decompositions}

The method of characteristic decomposition was introduced by Zheng {\it et al.} \cite{Chen1,Li1,Li-Zhang-Zheng,Li2,Li3,Li4} in inverting 2D pseudosteady rarefaction wave interactions.
 Chen and Qu \cite{CQ1} also used the method of characteristic decomposition to construct global solutions of 2D steady rarefaction wave interactions.
These results were obtained under the assumption that the flow is isentropic and irrotational.
In this paper, we are concerned with non-isentropic rotational flows. So, we need to derive some characteristic decompositions for the 2D steady full Euler system (\ref{PsEuler}). These decompositions
can be seen as systems of ``ordinary differential equations'' for some first order derivatives of the solution and will be extensively  used to control the bounds of the derivatives of the solution.

\begin{prop}
We have the commutator relations
\begin{equation}
\begin{array}{rcl}
\bar{\partial}_{0} \bar{\partial}_{+}- \bar{\partial}_{+} \bar{\partial}_{0}=
\displaystyle\frac{1}{\sin A}\Big[\big(\cos A \bar{\partial}_{+}\sigma- \bar{\partial}_{0}\alpha\big) \bar{\partial}_{0}-
\big( \bar{\partial}_{+}\sigma-\cos A \bar{\partial}_{0}\alpha\big) \bar{\partial}_{+}\Big]
\end{array}
\label{comm1}
\end{equation}
and
\begin{equation}
\begin{array}{rcl}
\bar{\partial}_{-} \bar{\partial}_{+}- \bar{\partial}_{+} \bar{\partial}_{-}=
\displaystyle\frac{1}{\sin2 A}\Big[\big(\cos2 A \bar{\partial}_{+}\beta- \bar{\partial}_{-}\alpha\big) \bar{\partial}_{-}-
\big( \bar{\partial}_{+}\beta-\cos2 A \bar{\partial}_{-}\alpha\big) \bar{\partial}_{+}\Big].
\end{array}
\label{comm}
\end{equation}
\end{prop}
\begin{proof}
The commutator relation (\ref{comm1}) was first proved by the author in \cite{Lai5}.
The commutator relation (\ref{comm}) was first given by Li and Zhang and Zheng in \cite{Li-Zhang-Zheng}.
For the sake of completeness, we sketch the proof.

From (\ref{tau}) and (\ref{72801}) we have
\begin{equation}\label{41907}
\partial_{x}=\frac{\sin\alpha\bar{\partial}_{0}-\sin\sigma\bar{\partial}_{+}}{\sin A}\quad \mbox{and}\quad
\partial_{y}=-\frac{\cos\alpha\bar{\partial}_{0}-\cos\sigma\bar{\partial}_{+}}{\sin A}.
\end{equation}
By computation, we have
$$
\begin{aligned}
&\bar{\partial}_{0} \bar{\partial}_{+}- \bar{\partial}_{+} \bar{\partial}_{0}\\~=~&(\cos\sigma\partial_{x}+\sin\sigma\partial_{y})(\cos\alpha\partial_{x}+\sin\alpha\partial_{y})
-(\cos\alpha\partial_{x}+\sin\alpha\partial_{y})(\cos\sigma\partial_{x}+\sin\sigma\partial_{y})
\\~=~&(\bar{\partial}_{0}\cos\alpha)\partial_{x}+(\bar{\partial}_{0}\sin\alpha)\partial_{y}
-(\bar{\partial}_{+}\cos\sigma)\partial_{x}-(\bar{\partial}_{+}\sin\sigma)\partial_{y}.
\end{aligned}
$$
Inserting (\ref{41907}) into this we can get (\ref{comm1}). The proof for (\ref{comm}) is similar.
\end{proof}

\begin{prop}For the derivative of the entropy, we have
\begin{equation}\label{4105}
\bar{\partial}_{0}\Big(\frac{\bar{\partial}_{+} s}{c^{\frac{\gamma+1}{\gamma-1}}}\Big)
~=~0.
\end{equation}
\end{prop}
\begin{proof}
From (\ref{192306}), (\ref{192308}), (\ref{32403}), (\ref{42001}), and (\ref{comm1})  we have
$$
\begin{aligned}
\bar{\partial}_{0}\bar{\partial}_{+} s&=\frac{1}{\sin A}\big(\cos A\bar{\partial}_{0}\alpha-
\bar{\partial}_{+}\sigma\big)\bar{\partial}_{+} s
\\&=\frac{1}{\sin A}\Big[\frac{1}{2}(\bar{\partial}_{+}\alpha+\bar{\partial}_{-}\alpha)
-\bar{\partial}_{+}\sigma\Big]\bar{\partial}_{+} s\\&=\frac{1}{2\sin A}\Big[\bar{\partial}_{-}\alpha
-\bar{\partial}_{+}\beta\Big]\bar{\partial}_{+} s\\&=\frac{1}{2c\sin A}(1+\kappa)\tan A\Big[\bar{\partial}_{-}c
+\bar{\partial}_{+}c\Big]\bar{\partial}_{+} s\\&=\Big(\frac{\gamma+1}{\gamma-1}\Big)\frac{\bar{\partial}_{0}c\bar{\partial}_{+} s}{c}.
\end{aligned}
$$
Thus, we have
$$
\bar{\partial}_{0}\Big(\frac{\bar{\partial}_{+} s}{c^{\frac{\gamma+1}{\gamma-1}}}\Big)
~=~c^{-\frac{\gamma+1}{\gamma-1}}\bar{\partial}_{0}\bar{\partial}_{+} s-\frac{\gamma+1}{\gamma-1}c^{\frac{-2\gamma}{\gamma-1}}\bar{\partial}_{0}c\bar{\partial}_{+} s~=~0.
$$

We then have this proposition.
\end{proof}

\begin{prop}For the vorticity, we have
\begin{equation}\label{72803}
\begin{aligned}
\bar{\partial}_{0}\Big(\frac{\omega}{\rho}\Big)
=-\Big(\frac{\bar{\partial}_{+}s}{c^{\frac{\gamma+1}{\gamma-1}}}\Big)\frac{(s\gamma)^{\frac{1}{\gamma-1}} }{\gamma(\gamma-1)s}\bar{\partial}_{0}c^2.
\end{aligned}
\end{equation}
\end{prop}
\begin{proof}
From the second and the third equations of (\ref{PsEuler}), we have
\begin{equation}\label{192301}
u\omega_x+v\omega_y+(u_x+v_y)\omega+\tau^{-\gamma}\big(\tau_y s_x-\tau_x s_y\big)=0.
\end{equation}

From (\ref{tau}), (\ref{42001}), and (\ref{192303})--(\ref{72802}) we have
\begin{equation}\label{41105}
\begin{aligned}
u_x+v_y&=\frac{1}{\sin 2A}\Big(\sin\alpha\bar{\partial}_{-}u-\sin\beta\bar{\partial}_{+}u+
\cos\beta\bar{\partial}_{+}v-\cos\alpha\bar{\partial}_{-}v\Big)\\&=
-\frac{\kappa}{\sin 2A}(\bar{\partial}_{+}c+\bar{\partial}_{-}c)\\&=-
\frac{\kappa}{\sin A}\bar{\partial}_{0}c
\end{aligned}
\end{equation}
and
\begin{equation}\label{41106}
\begin{aligned}
&\tau_y s_x-\tau_x s_y\\[4pt]
=&\Bigg\{\Big(\frac{\cos\alpha+\cos\beta}{\sin2A}\Big)\frac{\sin\beta\bar{\partial}_{+}\tau
-\sin\alpha\bar{\partial}_{-}\tau}{\sin2A}-
\Big(\frac{\sin\alpha+\sin\beta}{\sin2A}\Big)\frac{\cos\beta\bar{\partial}_{+}\tau
-\cos\alpha\bar{\partial}_{-}\tau}{\sin2A}\Bigg\}\bar{\partial}_{+}s
\\=&-\frac{\bar{\partial}_{+}s}{\sin 2A}(\bar{\partial}_{+}\tau+\bar{\partial}_{-}\tau)
\\=&-
\frac{\bar{\partial}_{+}s}{\sin A}\bar{\partial}_{0}\tau.
\end{aligned}
\end{equation}

Inserting (\ref{41105}) and (\ref{41106}) into (\ref{192301}), we get
\begin{equation}\label{4801}
\begin{aligned}
\bar{\partial}_{0}\omega&~=~\frac{\kappa\omega}{c}\bar{\partial}_{0}c+
\tau^{-\gamma}\frac{\bar{\partial}_{+}s}{c}\bar{\partial}_{0}\tau~=~
\frac{\kappa\omega}{c}\bar{\partial}_{0}c-2j\bar{\partial}_{+}s\frac{\bar{\partial}_{0}c}{c}.
\end{aligned}
\end{equation}
Consequently, we get
$$
\begin{aligned}
\bar{\partial}_{0}\Big(\frac{\omega}{\rho}\Big)&=\frac{1}{\rho}\bar{\partial}_{0}\omega
-\frac{\omega}{\rho^2}\bar{\partial}_{0}\rho\\&=
\frac{\kappa\omega}{\rho c}\bar{\partial}_{0}c-2j\bar{\partial}_{+}s\frac{\bar{\partial}_{0}c}{\rho c}-\frac{\omega}{\rho^2}\bar{\partial}_{0}\rho
\\&=-2j\bar{\partial}_{+}s\frac{\bar{\partial}_{0}c}{\rho c}~=~-\Big(\frac{\bar{\partial}_{+}s}{c^{\frac{\gamma+1}{\gamma-1}}}\Big)\frac{(s\gamma)^{\frac{1}{\gamma-1}} }{\gamma(\gamma-1)s}\bar{\partial}_{0}c^2.
\end{aligned}
$$

We then have this proposition.
\end{proof}

\begin{prop}\label{10604}
We have the characteristic decompositions:
\begin{equation}\label{cd}
\left\{
  \begin{array}{ll}
  \begin{array}{rcl}
 \displaystyle c\bar{\partial}_{-}\Big(\bar{\partial}_{+}c-\frac{j\bar{\partial}_{+}s}{\kappa}\Big)&=&\displaystyle
   \frac{(\gamma+1)\bar{\partial}_{+}c\bar{\partial}_{+}c}{2(\gamma-1)\cos^2A}+\frac{(\gamma+1)-2\sin^2 2A}{2(\gamma-1)\cos^2 A}
   \bar{\partial}_{-}c\bar{\partial}_{+}c\\[14pt]&&\displaystyle
+(H_{11}\bar{\partial}_{-}c+H_{12}\bar{\partial}_{+}c)\omega\sin^2A+(H_{13}\bar{\partial}_{-}c+H_{14}\bar{\partial}_{+}c)
j\bar{\partial}_{+}s\\
[8pt]&&\displaystyle +H_{15}(j\bar{\partial}_{+}s)^2+ H_{16}\omega\sin^2Aj\bar{\partial}_{+}s,
\end{array}
\\[38pt]
\begin{array}{rcl}
   \displaystyle c\bar{\partial}_{+}\Big(\bar{\partial}_{-}c-\frac{j\bar{\partial}_{-}s}{\kappa}\Big)&=&\displaystyle
\frac{(\gamma+1)\bar{\partial}_{-}c \bar{\partial}_{-}c}{2(\gamma-1)\cos^2A}+\frac{(\gamma+1)-2\sin^2 2A}{2(\gamma-1)\cos^2 A}
   \bar{\partial}_{+}c \bar{\partial}_{-}c\\[14pt]&&\displaystyle
+(H_{21}\bar{\partial}_{-}c+H_{22}\bar{\partial}_{+}c)\omega\sin^2A+(H_{23}\bar{\partial}_{-}c+H_{24}\bar{\partial}_{+}c)
j\bar{\partial}_{+}s\\
[8pt]&&\displaystyle +H_{25}(j\bar{\partial}_{+}s)^2+ H_{26}\omega\sin^2Aj\bar{\partial}_{+}s,
\end{array}
  \end{array}
\right.
\end{equation}
where $H_{ij}=\frac{{f_{ij}}}{\cos ^2 A}$, and ${f_{ij}}$ are polynomial forms in terms of $\sin A$.
\end{prop}

\begin{proof}
From (\ref{comm}) we have
$$
\begin{array}{rcl}
\bar{\partial}_{-} \bar{\partial}_{+}u- \bar{\partial}_{+} \bar{\partial}_{-}u=
\displaystyle\frac{1}{\sin2 A}\Big[\big(\cos2 A \bar{\partial}_{+}\beta- \bar{\partial}_{-}\alpha\big) \bar{\partial}_{-}u-
\big( \bar{\partial}_{+}\beta-\cos2 A \bar{\partial}_{-}\alpha\big) \bar{\partial}_{+}u\Big].
\end{array}
$$
Inserting (\ref{192303}) and (\ref{192304}) into this, we have
$$
\begin{array}{rcl}
&&\bar{\partial}_{-}\big[\kappa\sin\beta\bar{\partial}_{+}c+\omega\cos \sigma\sin A-j\sin \beta \bar{\partial}_{+}s]\\[6pt]&&\quad+
\bar{\partial}_{+}\big[\kappa\sin\alpha\bar{\partial}_{-}c+\omega\cos \sigma\sin A-j\sin \alpha \bar{\partial}_{-}s\big]
\\[6pt]&=&\displaystyle\frac{1}{\sin2 A}\Big[( \bar{\partial}_{-}\alpha- \cos2 A\bar{\partial}_{+}\beta)(\kappa\sin\alpha\bar{\partial}_{-}c+\omega\cos \sigma\sin A-j\sin \alpha \bar{\partial}_{-}s)\\[8pt]&&\displaystyle\qquad\quad-(\bar{\partial}_{+}\beta-\cos2 A \bar{\partial}_{-}\alpha)(\kappa\sin\beta\bar{\partial}_{+}c+\omega\cos \sigma\sin A-j\sin \beta \bar{\partial}_{+}s)\Big].
\end{array}
$$
Thus, we have
\begin{equation}\label{32401}
\begin{aligned}
&\kappa\sin\beta\bar{\partial}_{-}\bar{\partial}_{+}c+\kappa\sin\alpha\bar{\partial}_{+}\bar{\partial}_{-}c
+\kappa\cos\beta\bar{\partial}_{-}\beta\bar{\partial}_{+}c+\kappa\cos\alpha\bar{\partial}_{+}\alpha\bar{\partial}_{-}c
\\&+\bar{\partial}_{-}(\omega\cos \sigma\sin A)+
\bar{\partial}_{+}(\omega\cos \sigma\sin A)-\bar{\partial}_{-}\big(j\sin \beta \bar{\partial}_{+}s\big)
-\bar{\partial}_{+}\big(j\sin \alpha \bar{\partial}_{-}s\big)
\\~=~&\frac{\kappa}{\sin2 A}\Big[(\sin\alpha \bar{\partial}_{-}\alpha-\sin\alpha \cos2 A\bar{\partial}_{+}\beta)\bar{\partial}_{-}c-(\sin\beta\bar{\partial}_{+}\beta-\sin\beta\cos2 A \bar{\partial}_{-}\alpha)\bar{\partial}_{+}c\Big]\\&
-\frac{\omega\cos\sigma\sin A}{\sin2A}\Big[ \cos2 A\bar{\partial}_{+}\beta- \bar{\partial}_{-}\alpha+\bar{\partial}_{+}\beta-\cos2 A \bar{\partial}_{-}\alpha\Big]\\&+
\frac{1}{\sin2 A}\Big[(\cos2 A\bar{\partial}_{+}\beta- \bar{\partial}_{-}\alpha)j\sin \alpha \bar{\partial}_{-}s+(\bar{\partial}_{+}\beta-\cos2 A \bar{\partial}_{-}\alpha)j\sin \beta \bar{\partial}_{+}s\Big],
\end{aligned}
\end{equation}
Inserting
$$
\begin{array}{rcl}
\bar{\partial}_{+} \bar{\partial}_{-}c~=~\bar{\partial}_{-} \bar{\partial}_{+}c-
\displaystyle\frac{1}{\sin2 A}\Big[\big(\cos2 A \bar{\partial}_{+}\beta- \bar{\partial}_{-}\alpha\big) \bar{\partial}_{-}c-
\big( \bar{\partial}_{+}\beta-\cos2 A \bar{\partial}_{-}\alpha\big) \bar{\partial}_{+}c\Big]
\end{array}
$$
into (\ref{32401}), we get
\begin{equation}\label{192305}
\begin{aligned}
&(\sin\alpha+\sin\beta)c\bar{\partial}_{-}\bar{\partial}_{+}c\\~=~&
-\frac{c}{\sin2 A}\Big[(\sin\beta+\sin\alpha)\bar{\partial}_{+}\beta-(\sin\alpha+\sin\beta)\cos2 A \bar{\partial}_{-}\alpha+\cos\beta\sin2A\bar{\partial}_{-}\beta\Big]\bar{\partial}_{+}c\\&
-c\cos\alpha\bar{\partial}_{+}\alpha\bar{\partial}_{-}c
-\frac{c\omega\cos\sigma\sin A}{\kappa\sin2A}\Big[ \cos2 A\bar{\partial}_{+}\beta- \bar{\partial}_{-}\alpha+\bar{\partial}_{+}\beta-\cos2 A \bar{\partial}_{-}\alpha\Big]\\&-\frac{c}{\kappa}\Big[\bar{\partial}_{-}(\omega\cos \sigma\sin A)+
\bar{\partial}_{+}(\omega\cos \sigma\sin A)\Big]+\frac{c}{\kappa}\Big[\bar{\partial}_{-}\big(j\sin \beta \bar{\partial}_{+}s\big)
+\bar{\partial}_{+}\big(j\sin \alpha \bar{\partial}_{-}s\big)\Big]\\&+
\frac{c}{\kappa\sin2 A}\Big[(\cos2 A\bar{\partial}_{+}\beta- \bar{\partial}_{-}\alpha)j\sin \alpha \bar{\partial}_{-}s+(\bar{\partial}_{+}\beta-\cos2 A \bar{\partial}_{-}\alpha)j\sin \beta \bar{\partial}_{+}s\Big].
\end{aligned}
\end{equation}
In what follows, we are going to compute the right part of (\ref{192305}) item by item.

\vskip 6pt
{\it Part 1.} Using (\ref{192306})--(\ref{192309}), we have
\begin{equation}\label{4803}
\begin{aligned}
&-c\cos\alpha\bar{\partial}_{+}\alpha\bar{\partial}_{-}c\\
&-\frac{c}{\sin2 A}\Big[(\sin\beta+\sin\alpha)\bar{\partial}_{+}\beta-(\sin\alpha+\sin\beta)\cos2 A \bar{\partial}_{-}\alpha+\cos\beta\sin2A\bar{\partial}_{-}\beta\Big]\bar{\partial}_{+}c
\\~=~&\Big(\frac{1+\kappa}{2}\Big)\Omega\cos\alpha\sin 2A\bar{\partial}_{+}c\bar{\partial}_{-}c+\frac{1}{\sin2 A}\Big[(1+\kappa)(\sin\beta+\sin\alpha)\tan A\bar{\partial}_{+}c\\&+(1+\kappa)(\sin\alpha+\sin\beta)\tan A\cos2 A \bar{\partial}_{-}c-\Big(\frac{1+\kappa}{2}\Big)\Omega\cos\beta\sin^2 2A\bar{\partial}_{-}c\Big]\bar{\partial}_{+}c\\&+\omega\cos\alpha\sin^2A\tan A\bar{\partial}_{-}c
-\cos\alpha\tan A\cos 2A j\bar{\partial}_{+}s\bar{\partial}_{-}c
\\~&-\frac{\omega\sin^2A\tan A}{\sin2 A}\Big[(\sin\beta+\sin\alpha)-(\sin\alpha+\sin\beta)\cos2 A -\cos\beta\sin2A\Big]\bar{\partial}_{+}c
\\~&-\frac{j\tan A\bar{\partial}_{+}s}{\sin2 A}\Big[(\sin\beta+\sin\alpha)-(\sin\alpha+\sin\beta)\cos2 A +\cos\beta\sin2A\cos 2A\Big]\bar{\partial}_{+}c.
\end{aligned}
\end{equation}

By (\ref{tau}) and $\sin\alpha+\sin\beta=2\sin\sigma\cos A$, we have
\begin{equation}
\begin{aligned}
&\frac{1}{\sin\alpha+\sin\beta}\Bigg\{\Big(\frac{1+\kappa}{2}\Big)\Omega\cos\alpha\sin 2A\bar{\partial}_{+}c\bar{\partial}_{-}c+\frac{1}{\sin2 A}\Big[(1+\kappa)(\sin\beta+\sin\alpha)\tan A\bar{\partial}_{+}c\\&\qquad\qquad+(1+\kappa)(\sin\alpha+\sin\beta)\tan A\cos2 A \bar{\partial}_{-}c-\Big(\frac{1+\kappa}{2}\Big)\Omega\cos\beta\sin^2 2A\bar{\partial}_{-}c\Big]\bar{\partial}_{+}c\Bigg\}\\~=~&\frac{(\gamma+1)\bar{\partial}_{+}c\bar{\partial}_{+}c}{2(\gamma-1)\cos^2A}+
\Bigg(\frac{(1+\kappa)\cos 2A }{2\cos^2 A}-(1+\kappa)\Omega \sin^2 A\Bigg)\bar{\partial}_{+}c\bar{\partial}_{-}c
\\~=~&\frac{(\gamma+1)\bar{\partial}_{+}c\bar{\partial}_{+}c}{2(\gamma-1)\cos^2A}
+\Bigg(\frac{\gamma-3}{\gamma-1}\sin^2 A+\frac{(\gamma+1)\sin^4 A}{(\gamma-1)\cos^2 A}+\frac{(\gamma+1)\cos 2 A}{2(\gamma-1)\cos^2 A}\Bigg)\bar{\partial}_{+}c\bar{\partial}_{-}c
\\ ~=~&\frac{(\gamma+1)\bar{\partial}_{+}c\bar{\partial}_{+}c}{2(\gamma-1)\cos^2A}+\frac{(\gamma+1)-2\sin^2 2A}{2(\gamma-1)\cos^2 A}
   \bar{\partial}_{-}c\bar{\partial}_{+}c.
\end{aligned}
\end{equation}

From (\ref{tau}), we also have
\begin{equation}
\begin{aligned}
&(\sin\beta+\sin\alpha)-(\sin\alpha+\sin\beta)\cos2 A -\cos\beta\sin2A\\
=&\sin\beta-\sin\alpha\cos2 A=-\cos \alpha\sin 2A
\end{aligned}
\end{equation}
and
\begin{equation}
\begin{aligned}
&(\sin\beta+\sin\alpha)-(\sin\alpha+\sin\beta)\cos2 A +\cos\beta\sin2A\cos 2A\\~=~&
(\sin\beta+\sin\alpha)+\cos 2A(\cos\beta\sin2A-\sin\alpha-\sin\beta)
\\~=~&
(\sin\beta+\sin\alpha)+\cos 2A(-\sin\beta\cos2A-\sin\beta)
\\~=~& \sin\beta\sin^2 2A+ \sin 2A \cos \beta ~=~(\sin\beta+\sin\alpha)\sin^2 2A +\sin 2A \cos \alpha\cos 2A.
\end{aligned}
\end{equation}
Thus, we have
\begin{equation}
\begin{aligned}
&\omega\cos\alpha\sin^2A\tan A\bar{\partial}_{-}c
\\&-\frac{\omega\sin^2A\tan A}{\sin2 A}\Big[(\sin\beta+\sin\alpha)-(\sin\alpha+\sin\beta)\cos2 A -\cos\beta\sin2A\Big]\bar{\partial}_{+}c\\~=~&\omega\cos\alpha\sin^2A\tan A(\bar{\partial}_{-}c+\bar{\partial}_{+}c)
\\~=~&\omega\cos\sigma\cos A\sin^2A\tan A(\bar{\partial}_{-}c+\bar{\partial}_{+}c)-\omega\sin\sigma\sin A\sin^2A\tan A(\bar{\partial}_{-}c+\bar{\partial}_{+}c)
\end{aligned}
\end{equation}
and
\begin{equation}
\begin{aligned}
&-\cos\alpha\tan A\cos 2A j\bar{\partial}_{+}s\bar{\partial}_{-}c
\\~&-\frac{j\tan A\bar{\partial}_{+}s}{\sin2 A}\Big[(\sin\beta+\sin\alpha)-(\sin\alpha+\sin\beta)\cos2 A +\cos\beta\sin2A\cos 2A\Big]\bar{\partial}_{+}c\\~=~&-\cos\alpha\tan A\cos 2A j\bar{\partial}_{+}s(\bar{\partial}_{-}c+\bar{\partial}_{+}c)
-\tan A\sin2 A(\sin\beta+\sin\alpha)j\bar{\partial}_{+}s\bar{\partial}_{+}c\\~=~&
-\cos\sigma\sin A\cos 2A j\bar{\partial}_{+}s(\bar{\partial}_{-}c+\bar{\partial}_{+}c)+\sin\sigma\sin A\tan A\cos 2A j\bar{\partial}_{+}s(\bar{\partial}_{-}c+\bar{\partial}_{+}c)\\&-\tan A\sin2 A(\sin\beta+\sin\alpha)j\bar{\partial}_{+}s\bar{\partial}_{+}c.
\end{aligned}
\end{equation}

{\it Part 2.} Using (\ref{192306}) and (\ref{192308}), we have
\begin{equation}
\begin{aligned}
&-\frac{c\omega\cos\sigma\sin A}{\kappa\sin2A}\Big[ \cos2 A\bar{\partial}_{+}\beta- \bar{\partial}_{-}\alpha+\bar{\partial}_{+}\beta-\cos2 A \bar{\partial}_{-}\alpha\Big]\\~=~&-\frac{(1+\kappa)\omega\cos\sigma\sin A\tan A}{\kappa\sin2A}\big[ -\cos2 A\bar{\partial}_{+}c- \bar{\partial}_{-}c-\bar{\partial}_{+}c-\cos2 A \bar{\partial}_{-}c\big]\\~=~&
\frac{(1+\kappa)\omega\cos\sigma\sin A}{\kappa}(\bar{\partial}_{-}c+\bar{\partial}_{+}c).
\end{aligned}
\end{equation}

{\it Part 3.} Using (\ref{4801}), we have
\begin{equation}\label{32404}
\begin{aligned}
&-\frac{c}{\kappa}\cos \sigma\sin A\big(\bar{\partial}_{-}\omega+
\bar{\partial}_{+}\omega\big)\\~=~&
-\frac{2c}{\kappa}\cos \sigma\sin A\cos A\bar{\partial}_0\omega\\~=~&
-\frac{2c}{\kappa}\cos \sigma\sin A\cos A\left[\frac{\kappa}{2c\cos A}(\bar{\partial}_{+}c+\bar{\partial}_{-}c)\omega-\frac{j\bar{\partial}_{+}s}{c\cos A}(\bar{\partial}_{+}c+\bar{\partial}_{-}c)\right]
\\~=~&-\cos \sigma\sin A(\bar{\partial}_{+}c+\bar{\partial}_{-}c)\omega+\frac{2}{\kappa}\cos \sigma\sin A
(\bar{\partial}_{+}c+\bar{\partial}_{-}c)j\bar{\partial}_{+}s.
\end{aligned}
\end{equation}

By a direct computation, we have
\begin{equation}\label{4902}
\begin{aligned}
&(1+\kappa)\tan A+\Big(\frac{1+\kappa}{2}\Big)\Omega\sin 2A\\=&(1+\kappa)\tan A
+\Big(\frac{1+\kappa}{2}\Big)\Big(\frac{\kappa-1}{\kappa+1}-\tan^2 A\Big)\sin 2A=2\kappa \sin A \cos A
\end{aligned}
\end{equation}
and
$$
(1+\kappa)\tan A-\Big(\frac{1+\kappa}{2}\Big)\Omega\sin 2A=2(1+\kappa)\tan A-2\kappa \sin A \cos A.
$$
Therefore, by (\ref{192306})--(\ref{192309}) we have
\begin{equation}
\begin{aligned}
&-\frac{c\omega}{\kappa}\big[\bar{\partial}_{-}(\cos \sigma\sin A)+
\bar{\partial}_{+}(\cos \sigma\sin A)\big]\\~=~&
\frac{c\omega\sin\sigma\sin A}{\kappa}\big[\bar{\partial}_{-}\sigma+
\bar{\partial}_{+}\sigma\big]-\frac{c\omega\cos\sigma\cos A}{\kappa}\big[\bar{\partial}_{-}A+
\bar{\partial}_{+}A\big]
\\~=~&
\frac{c\omega\sin\sigma\sin A}{2\kappa}\big[\bar{\partial}_{-}(\alpha+\beta)+
\bar{\partial}_{+}(\alpha+\beta)\big]-\frac{c\omega\cos\sigma\cos A}{2\kappa}\big[\bar{\partial}_{-}(\alpha-\beta)+
\bar{\partial}_{+}(\alpha-\beta)\big]
\\~=~&
\frac{\omega\sin\sigma\sin A}{2\kappa}\Big[(1+\kappa)\tan A+\Big(\frac{1+\kappa}{2}\Big)\Omega\sin 2A\Big](\bar{\partial}_{-}c-\bar{\partial}_{+}c)\\&+\frac{\omega\sin\sigma\sin A(1+\cos2A)\tan A}{\kappa}j\bar{\partial}_{+}s\\&-
\frac{\omega\cos\sigma\cos A}{2\kappa}\Big[(1+\kappa)\tan A-\Big(\frac{1+\kappa}{2}\Big)\Omega\sin 2A\Big](\bar{\partial}_{-}c+\bar{\partial}_{+}c)
\\~=~&
\omega\sin\sigma\sin^2 A\cos A(\bar{\partial}_{-}c-\bar{\partial}_{+}c)+\frac{2\omega\sin\sigma\sin A\cos^2A\tan A}{\kappa}j\bar{\partial}_{+}s\\&+\omega\Big[-
\frac{(1+\kappa)\cos\sigma\sin A}{\kappa}+\cos\sigma\cos^2 A \sin A\Big](\bar{\partial}_{-}c+\bar{\partial}_{+}c)
\end{aligned}
\end{equation}

{\it Part 4.} By (\ref{comm}), (\ref{32402}), and (\ref{42001}) we have
\begin{equation}\label{42101}
\begin{aligned}
&\frac{c}{\kappa}\big[\bar{\partial}_{-}\big(j\sin \beta \bar{\partial}_{+}s\big)
+\bar{\partial}_{+}\big(j\sin \alpha \bar{\partial}_{-}s\big)\big]\\~=~&
\frac{c}{\kappa}\Big[\sin \beta\bar{\partial}_{-}j \bar{\partial}_{+}s+\sin \alpha\bar{\partial}_{+}j\bar{\partial}_{-}s
+j\cos\beta\bar{\partial}_{-}\beta\bar{\partial}_{+}s+j\cos\alpha\bar{\partial}_{+}\alpha\bar{\partial}_{-}s\\&
\quad +j\sin\beta\bar{\partial}_{-}\bar{\partial}_{+}s+j\sin\alpha\bar{\partial}_{+}\bar{\partial}_{-}s+\sin \alpha\bar{\partial}_{-}j\bar{\partial}_{+}s+\sin \alpha\bar{\partial}_{-}j\bar{\partial}_{-}s\Big]
\\~=~&-\frac{j}{\kappa}\sin\alpha(\bar{\partial}_{+}c+\bar{\partial}_{-}c)\bar{\partial}_{+}s+\frac{c}{\kappa}
(\sin\alpha+\sin\beta)\bar{\partial}_{-}j\bar{\partial}_{+}s\\&~+
\frac{c}{\kappa}\Big[j\cos\beta\bar{\partial}_{-}\beta\bar{\partial}_{+}s+j\cos\alpha\bar{\partial}_{+}
\alpha\bar{\partial}_{-}s\Big]+\frac{c}{\kappa}(\sin\alpha+\sin\beta)j\bar{\partial}_{-}\bar{\partial}_{+}s\\&
-\frac{cj\sin\alpha}{\kappa\sin2 A}\Big[(\cos2 A\bar{\partial}_{+}\beta- \bar{\partial}_{-}\alpha) \bar{\partial}_{-}s-(\bar{\partial}_{+}\beta-\cos2 A \bar{\partial}_{-}\alpha) \bar{\partial}_{+}s\Big]\\~=~&
-\frac{j}{\kappa}(\sin\sigma\cos A+\cos\sigma\sin A)(\bar{\partial}_{+}c+\bar{\partial}_{-}c)\bar{\partial}_{+}s+
\frac{c}{\kappa}(\sin\alpha+\sin\beta)\bar{\partial}_{-}\big(j\bar{\partial}_{+}s\big)\\&+
\frac{c}{\kappa}\Big[j\cos\beta\bar{\partial}_{-}\beta\bar{\partial}_{+}s+j\cos\alpha\bar{\partial}_{+}
\alpha\bar{\partial}_{-}s\Big]\\&-\frac{cj\sin\alpha}{\kappa\sin2 A}\Big[(\cos2 A\bar{\partial}_{+}\beta- \bar{\partial}_{-}\alpha) \bar{\partial}_{-}s-(\bar{\partial}_{+}\beta-\cos2 A \bar{\partial}_{-}\alpha) \bar{\partial}_{+}s\Big].
\end{aligned}
\end{equation}

Noticing the last parts of (\ref{192305}) and (\ref{42101}) and using (\ref{192306}) and (\ref{192308}), we have
\begin{equation}
\begin{aligned}
&\frac{cj}{\kappa\sin2 A}(\bar{\partial}_{+}\beta-\cos2 A \bar{\partial}_{-}\alpha) \bar{\partial}_{+}s\\=&\frac{j\bar{\partial}_{+}s}{\kappa\cos ^2 A}
\Bigg\{-\frac{1+\kappa}{2}(\bar{\partial}_{+}c+\cos 2A \bar{\partial}_{-}c)+(1-\cos 2A)\frac{\omega\sin^2 A}{2}+(1-\cos 2A)\frac{j\bar{\partial}_{+}s}{2}
\Bigg\}.
\end{aligned}
\end{equation}
Using  (\ref{192307}) and (\ref{192309}), we have
\begin{equation}\label{4101}
\begin{aligned}
&\frac{cj}{\kappa}\Big[\cos\beta\bar{\partial}_{-}\beta\bar{\partial}_{+}s+\cos\alpha\bar{\partial}_{+}
\alpha\bar{\partial}_{-}s\Big]\\~=~&\frac{cj\bar{\partial}_{+}s}{\kappa}\Big[
\cos\beta\bar{\partial}_{-}\beta-\cos\alpha\bar{\partial}_{+}
\alpha\Big]
\\~=~&\frac{cj\bar{\partial}_{+}s}{\kappa}\cos\sigma\cos A\big[
\bar{\partial}_{-}\beta-\bar{\partial}_{+}\alpha\big]+\frac{cj\bar{\partial}_{+}s}{\kappa}\sin\sigma\sin A\big[
\bar{\partial}_{-}\beta+\bar{\partial}_{+}
\alpha\big]\\~=~&
\frac{j\bar{\partial}_{+}s}{\kappa}\cos\sigma\cos A\Big(\frac{1+\kappa}{2}\Big)\Big(\frac{\kappa-1}{\kappa+1}-\tan^2 A\Big)\sin 2A\big[
\bar{\partial}_{-}c+\bar{\partial}_{+}c\big]\\&+
\frac{j\bar{\partial}_{+}s}{\kappa}\sin\sigma\sin A\Big(\frac{1+\kappa}{2}\Big)\Big(\frac{\kappa-1}{\kappa+1}-\tan^2 A\Big)\sin 2A\big[
\bar{\partial}_{-}c-\bar{\partial}_{+}c\big]\\&+\frac{j\bar{\partial}_{+}s}{\kappa}\sin\sigma\sin A\Big[-2\omega\sin^2A\tan A+2j\tan A \cos 2A \bar{\partial}_{+}s\Big]
\\~=~&
\left[\frac{(\kappa-1)j\bar{\partial}_{+}s}{\kappa}\cos\sigma\sin A\cos^2 A-\frac{(\kappa+1)j\bar{\partial}_{+}s}{\kappa}\cos\sigma\sin^3 A\right]\big(
\bar{\partial}_{-}c+\bar{\partial}_{+}c\big)\\&+
\frac{j\bar{\partial}_{+}s}{\kappa}\sin\sigma\sin A\Big(\frac{1+\kappa}{2}\Big)\Big(\frac{\kappa-1}{\kappa+1}-\tan^2 A\Big)\sin 2A\big[
\bar{\partial}_{-}c-\bar{\partial}_{+}c\big]\\&+\frac{j\bar{\partial}_{+}s}{\kappa}\sin\sigma\sin A\Big[-2\omega\sin^2A\tan A+2j\tan A \cos 2A \bar{\partial}_{+}s\Big].
\end{aligned}
\end{equation}


Therefore, combining with (\ref{192305})--(\ref{4101}) and using $\sin\alpha+\sin\beta=2\sin\sigma\cos A$, we can get the first equation of (\ref{cd}). (Remark: all the parts that contain $\cos\sigma$ can be eliminated.)

The proof for the other is similar; we omit the details. We then have Proposition \ref{10604}.
\end{proof}

In terms of the variables $R_{\pm}$ defined in (\ref{1973101}),
system (\ref{cd}) can be written as
\begin{equation}\label{cd6}
\left\{
  \begin{array}{ll}
  \begin{array}{rcl}
 \displaystyle c\bar{\partial}_{-}R_{+}&=&\displaystyle
   \frac{(\gamma+1)R_{+}^2}{2(\gamma-1)\cos^2A}+\frac{(\gamma+1)-2\sin^2 2A}{2(\gamma-1)\cos^2 A}
   R_{+}R_{-}\\[14pt]&&\displaystyle
+(\widetilde{H}_{11}R_{-}+\widetilde{H}_{12}R_{+})\omega\sin^2A+(\widetilde{H}_{13}R_{-}+\widetilde{H}_{14}R_{+})
j\bar{\partial}_{+}s\\
[8pt]&&\displaystyle +\widetilde{H}_{15}(j\bar{\partial}_{+}s)^2+ \widetilde{H}_{16}\omega\sin^2Aj\bar{\partial}_{+}s,
\end{array}
\\[38pt]
\begin{array}{rcl}
   \displaystyle c\bar{\partial}_{+}R_{-}&=&\displaystyle
\frac{(\gamma+1)R_{-}^2}{2(\gamma-1)\cos^2A}+\frac{(\gamma+1)-2\sin^2 2A}{2(\gamma-1)\cos^2 A}
   R_{-}R_{+}\\[14pt]&&\displaystyle
+(\widetilde{H}_{21}R_{-}+\widetilde{H}_{22}R_{+})\omega\sin^2A+(\widetilde{H}_{23}R_{-}+\widetilde{H}_{24}R_{+})
j\bar{\partial}_{+}s\\
[8pt]&&\displaystyle +\widetilde{H}_{25}(j\bar{\partial}_{+}s)^2+ \widetilde{H}_{26}\omega\sin^2Aj\bar{\partial}_{+}s,
\end{array}
  \end{array}
\right.
\end{equation}
where $\widetilde{H}_{ij}=\frac{\widetilde{f_{ij}}}{\cos ^2 A}$, and $\widetilde{f_{ij}}$ are polynomial forms in terms of $\sin A$.

Define
\begin{equation}
\delta_1:=\frac{\omega}{\rho}\quad \mbox{and}\quad \delta_2:=\frac{\bar{\partial}_{+}s}{c^{\frac{\gamma+1}{\gamma-1}}}.
\end{equation}
Then we have
$$
\omega\sin^2A=\frac{ \delta_1 c^{\frac{2\gamma}{\gamma-1}}}{q^2(\gamma s)^{\frac{1}{\gamma-1}}}\quad\mbox{and}\quad
j\bar{\partial}_{+}s=\frac{\delta_2 c^{\frac{2\gamma}{\gamma-1}}}{\gamma(\gamma-1)s}.
$$
Thus, (\ref{cd6}) can be written as
\begin{equation}\label{cd8}
\left\{
  \begin{array}{ll}
  \begin{array}{rcl}
 \displaystyle c\bar{\partial}_{-}R_{+}&=&\displaystyle
   \frac{(\gamma+1)R_{+}^2}{2(\gamma-1)\cos^2A}+\frac{(\gamma+1)-2\sin^2 2A}{2(\gamma-1)\cos^2 A}
   R_{+}R_{-}\\[12pt]&&\displaystyle
+(\widetilde{H}_{11}R_{-}+\widetilde{H}_{12}R_{+})
\frac{ \delta_1 c^{\frac{2\gamma}{\gamma-1}}}{q^2(\gamma s)^{\frac{1}{\gamma-1}}}+(\widetilde{H}_{13}R_{-}+\widetilde{H}_{14}R_{+})
\frac{\delta_2 c^{\frac{2\gamma}{\gamma-1}}}{\gamma(\gamma-1)s}\\
[8pt]&&\displaystyle +\widetilde{H}_{15}
\frac{\delta_2^2 c^{\frac{4\gamma}{\gamma-1}}}{\gamma^2(\gamma-1)^2s^2}+ \widetilde{H}_{16}\frac{\delta_1
\delta_2 c^{\frac{4\gamma}{\gamma-1}}}{\gamma(\gamma-1)sq^2(\gamma s)^{\frac{1}{\gamma-1}}},
\end{array}
\\[56pt]
\begin{array}{rcl}
   \displaystyle c\bar{\partial}_{+}R_{-}&=&\displaystyle
\frac{(\gamma+1)R_{-}^2}{2(\gamma-1)\cos^2A}+\frac{(\gamma+1)-2\sin^2 2A}{2(\gamma-1)\cos^2 A}
   R_{-}R_{+}\\[12pt]&&\displaystyle
+(\widetilde{H}_{21}R_{-}+\widetilde{H}_{22}R_{+})\frac{\delta_1 c^{\frac{2\gamma}{\gamma-1}}}{q^2(\gamma s)^{\frac{1}{\gamma-1}}}+(\widetilde{H}_{23}R_{-}+\widetilde{H}_{24}R_{+})
\frac{\delta_2 c^{\frac{2\gamma}{\gamma-1}}}{\gamma(\gamma-1)s}\\
[8pt]&&\displaystyle +\widetilde{H}_{25}
\frac{\delta_2^2 c^{\frac{4\gamma}{\gamma-1}}}{\gamma^2(\gamma-1)^2s^2}+ \widetilde{H}_{26}\frac{\delta_1
\delta_2 c^{\frac{4\gamma}{\gamma-1}}}{\gamma(\gamma-1)sq^2(\gamma s)^{\frac{1}{\gamma-1}}}.
\end{array}
  \end{array}
\right.
\end{equation}

\begin{prop}
For any constant $\nu>0$, we have the decompositions:
\begin{equation}\label{cd10}
\left\{
  \begin{array}{ll}
  \begin{array}{rcl}
 \displaystyle c\bar{\partial}_{-}\Big(\frac{R_{+}}{c^{\nu}}\Big)&=&\displaystyle
  \frac{1}{c^{\nu}}\Bigg\{ \frac{(\gamma+1)R_{+}^2}{2(\gamma-1)\cos^2A}+\frac{(\gamma+1)-2\sin^2 2A}{2(\gamma-1)\cos^2 A}
   R_{+}R_{-}-\nu R_{+}R_{-}\\[14pt]&&\displaystyle
+(\widehat{H}_{11}R_{-}+\widehat{H}_{12}R_{+})\frac{\delta_1 c^{\frac{2\gamma}{\gamma-1}}}{q^2(\gamma s)^{\frac{1}{\gamma-1}}}+(\widehat{H}_{13}R_{-}+\widehat{H}_{14}R_{+})
\frac{\delta_2 c^{\frac{2\gamma}{\gamma-1}}}{\gamma(\gamma-1)s}\\
[8pt]&&\displaystyle+\widehat{H}_{15}
\frac{\delta_2^2 c^{\frac{4\gamma}{\gamma-1}}}{\gamma^2(\gamma-1)^2s^2}+ \widehat{H}_{16}\frac{\delta_1
\delta_2 c^{\frac{4\gamma}{\gamma-1}}}{\gamma(\gamma-1)sq^2(\gamma s)^{\frac{1}{\gamma-1}}}\Bigg\},
\end{array}
\\[56pt]
\begin{array}{rcl}
   \displaystyle c\bar{\partial}_{+}\Big(\frac{R_{-}}{c^{\nu}}\Big)&=&\displaystyle
\frac{1}{c^{\nu}}\Bigg\{\frac{(\gamma+1)R_{-}^2}{2(\gamma-1)\cos^2A}+\frac{(\gamma+1)-2\sin^2 2A}{2(\gamma-1)\cos^2 A}
   R_{-}R_{+}-\nu R_{+}R_{-}\\[14pt]&&\displaystyle
+(\widehat{H}_{21}R_{-}+\widehat{H}_{22}R_{+})\frac{\delta_1 c^{\frac{2\gamma}{\gamma-1}}}{q^2(\gamma s)^{\frac{1}{\gamma-1}}}+(\widehat{H}_{23}R_{-}+\widehat{H}_{24}R_{+})
\frac{\delta_2 c^{\frac{2\gamma}{\gamma-1}}}{\gamma(\gamma-1)s}\\
[8pt]&&\displaystyle+\widehat{H}_{25}
\frac{\delta_2^2 c^{\frac{4\gamma}{\gamma-1}}}{\gamma^2(\gamma-1)^2s^2}+ \widehat{H}_{26}\frac{\delta_1
\delta_2 c^{\frac{4\gamma}{\gamma-1}}}{\gamma(\gamma-1)sq^2(\gamma s)^{\frac{1}{\gamma-1}}}\Bigg\},
\end{array}
  \end{array}
\right.
\end{equation}
where $\widehat{H}_{ij}=\frac{\widehat{f_{ij}}}{\cos ^2 A}$, and $\widehat{f_{ij}}$ are polynomial forms in terms of $\sin A$.
\end{prop}
\begin{proof}
The characteristic decompositions (\ref{cd6}) can be obtained directly by (\ref{cd8}). We omit the details.
\end{proof}

\section{\bf Global supersonic flows in the semi-infinite divergent duct}

\subsection{Simple waves adjacent to the constant state $(u_0, 0, \rho_0, s_0)$}
If the incoming supersonic flow is a uniform flow $(u_0, 0, \rho_0, s_0)$, then by the result of Courant and Friedrichs \cite{CF} (see Chapter IV.B) we know that there is a simple wave $\it S_{+}$  ($\it S_{-}$, resp.) with straight $C_{+}$  ($C_{-}$, resp.) characteristics issuing from lower wall $W_{-}$ (upper wall $W_{+}$, resp.); see Figure \ref{Fig25} (right).

Let
$$
A_0=\arcsin\frac{c_0}{u_0}, \quad c_*^2=\mu^2u_0^2+(1-\mu^2)c_0^2
,\quad \theta_*=A_0-\frac{\pi}{2}+\frac{1}{\mu}\arctan\big(\mu\sqrt{(u_0/c_0)^2-1}\big),$$
and
$$
\left\{
  \begin{array}{ll}
  \displaystyle\frac{\bar{u}(\theta)}{c_{*}}=\cos\mu(\theta-\theta_*)\cos \theta+\mu^{-1}\sin\mu (\theta-\theta_*)\sin \theta,\\[10pt]
\displaystyle\frac{\bar{v}(\theta)}{c_*}=\cos\mu (\theta-\theta_*)\sin \theta-\mu^{-1}\sin\mu (\theta-\theta_*)\cos \theta.
  \end{array}
\right.
$$
Then $u=\bar{u}(\theta)$, $v=\bar{v}(\theta)$ ($\theta_*-\frac{\pi}{2\mu}<\theta<\theta_*$) represents an epicycloidal characteristic on the $(u, v)$-plane; see Figure \ref{Fig25} (left).
Moreover, by a direct computation we have that along the epicycloidal characteristic,
$$
\begin{aligned}
&c=\bar{c}(\theta):=c_*\cos\mu(\theta-\theta_*),\quad q=\bar{q}(\theta):=c_*\sqrt{\cos^2\mu(\theta-\theta_*)+\mu^{-2}\sin^2\mu(\theta-\theta_*)}~ \\&\alpha=\bar{\alpha}(\theta):=\theta+\frac{\pi}{2},\quad A=\bar{A}(\theta):=\arcsin\left(\frac{\bar{c}(\theta)}{\bar{q}(\theta)}\right),\quad
\sigma=\bar{\sigma}(\theta):=\bar{\alpha}(\theta)-\bar{A}(\theta).\\ & \beta=\bar{\beta}(\theta):=\bar{\sigma}(\theta)-\bar{A}(\theta).\quad
\end{aligned}
$$
Let $\theta=\bar{\theta}(\sigma)$ be the inverse function of $\sigma=\bar{\sigma}(\theta)$ and let $\tilde{\theta}(\xi)=
\bar{\theta}\big(-\arctan f'(\xi)\big)$ ($\xi\geq 0$).
Then,
$$
\begin{aligned}
&\qquad (u, v, c)(x, y)=(\bar{u}, \bar{v}, \bar{c})\big(\tilde{\theta}(\xi)\big), \quad s(x,y)=s_0,\\[4pt]
&x=\xi+r\cos \bar{\alpha}\big(\tilde{\theta}(\xi)\big), \quad y=-f(\xi)+r\sin\bar{\alpha}\big(\tilde{\theta}(\xi)\big), \quad \xi\geq 0, \quad r\geq 0,
\end{aligned}
$$
is a simple wave with straight $C_{+}$ characteristic lines issuing from $W_{-}$.
By symmetry, we can construct the simple wave with straight $C_{-}$ characteristic lines issuing from $W_{+}$.
These results are classical, so we omit the details.

The two simple waves start to interact from a point $\bar{P}=(f(0)\cot A_0, 0)$.
Through $\bar{P}$ draw a forward $C_{-}$ ($C_{+}$, resp.) cross characteristic curve $C_{-}^{\bar{P}}$ ($C_{+}^{\bar{P}}$, resp.) in
$\Delta_{+}$ ($\Delta_{-}$, resp.).
In the present paper, $(u_0, 0, \rho_0, s_0)$ and the duct $\Sigma$ are required to satisfy that the characteristic curve $C_{-}^{\bar{P}}$ ($C_{+}^{\bar{P}}$, resp.) meets the lower wall $W_{-}$ (upper wall $W_{+}$, resp.) at a point $\bar{B}_0$ ($\bar{D}_0$, resp.) and $c(\bar{B}_0)>0$, as shown in Figure \ref{Fig25}(right).

\begin{rem}
It is obvious that if $f(x)\equiv 0$ then this case can happen.
So, by studying an initial value problem for an ordinary differential equation, one can find a $\eta>0$ depending only on $(u_0, 0, \rho_0, s_0)$ and $f(0)$,
 such that if $0<f''(x)<\eta$ as $x\geq 0$ then the above situation can occur.
\end{rem}

From (\ref{192306})--(\ref{192309}) and (\ref{4902}), we have
\begin{equation}\label{42201}
\begin{aligned}
\bar{\partial}_{0}\sigma&~=~\Big(\frac{\bar{\partial}_{+}+\bar{\partial}_{-}}{2\cos A}\Big)\Big(\frac{\alpha+\beta}{2}\Big)~=~\frac{1}{4c\cos A}(c\bar{\partial}_{+}+c\bar{\partial}_{-})(\alpha+\beta)\\[10pt]&~=
\frac{\bar{\partial}_{-}c-\bar{\partial}_{+}c}{(\gamma-1)q}+\frac{j}{q}\bar{\partial}_{+}s
~=~\frac{R_{-}-R_{+}}{(\gamma-1)q} \qquad \mbox{on}\quad W_{-}.
\end{aligned}
\end{equation}

Since simple waves are isentropic irrotational, the simple wave $\it S_{+}$ satisfies
$$
\begin{aligned}
\bar{\partial}_{-}c&=(\gamma-1)q \bar{\partial}_{0}\sigma
=-\frac{(\gamma-1)q f''(x)}{\big(1+[f'(x)]^2\big)^{\frac{3}{2}}}<0\quad \mbox{on}\quad W_{-}
\end{aligned}
$$
and
$$
 c\bar{\partial}_{+}\bar{\partial}_{-}c=
\frac{(\gamma+1)(\bar{\partial}_{-}c)^2 }{2(\gamma-1)\cos^2A}.
$$
Thus, the simple wave solution $\it S_{+}$ satisfies
\begin{equation}
\begin{aligned}
&\quad\sup\limits_{\widehat{\bar{P}\bar{B}_0}}A=A_0,~\quad\sup\limits_{\widehat{\bar{P}\bar{B}_0}}c=c_0, ~\quad~ \inf\limits_{\widehat{\bar{P}\bar{B}_0}}c=c(\bar{B}_0)>0,\\&  \inf\limits_{\widehat{\bar{P}\bar{B}_0}}q=u_0,~\quad
 \sup\limits_{\widehat{\bar{P}\bar{B}_0}}\bar{\partial}_{-}c
<0,~\quad\mbox{and}\quad
\inf\limits_{\widehat{\bar{P}\bar{B}_0}}\bar{\partial}_{-}c>-\infty.
\end{aligned}
\end{equation}
By symmetry, the simple wave solution $\it S_{-}$ satisfies
\begin{equation}
\begin{aligned}
&\quad\sup\limits_{\widehat{\bar{P}\bar{D}_0}}A=A_0,~\quad~\sup\limits_{\widehat{\bar{P}\bar{D}_0}}c=c_0, ~\quad~ \inf\limits_{\widehat{\bar{P}\bar{D}_0}}c=c(\bar{D}_0)>0,\\&  \inf\limits_{\widehat{\bar{P}\bar{D}_0}}q=u_0,~\quad
 \sup\limits_{\widehat{\bar{P}\bar{D}_0}}\bar{\partial}_{+}c
<0,~\quad\mbox{and}\quad
\inf\limits_{\widehat{\bar{P}\bar{D}_0}}\bar{\partial}_{+}c>-\infty.
\end{aligned}
\end{equation}
Therefore, we can define the following constants:
\begin{equation}
\overline{c}_m:=c(\bar{B}_0), \quad
\overline{m}_1:=-\sup\limits_{\widehat{\bar{P}\bar{B}_0}}\bar{\partial}_{-}c,
\quad \mbox{and}\quad \overline{M}_1:
=-\inf\limits_{\widehat{\bar{P}\bar{B}_0}}\bar{\partial}_{-}c.
\end{equation}
It is obvious that these constants depend only on the state $(u_0, 0, \rho_0, s_0)$ and the duct $\Sigma$.

\subsection{Some useful constants}
We shall define some constants that depend only on the state $(u_0, 0, \rho_0, s_0)$ and the duct $\Sigma$. These constants will be extensively used in establishing the a priori $C^1$ norm estimates of the solution.

Let $c_m$ be a constant in $(0, \frac{\overline{c}_m}{2}]$ such that when $c<c_m$ there holds
\begin{equation}\label{72602}
\frac{\gamma+1}{(\gamma-1)}\left(1-\frac{c^2}{2\big(\frac{\widehat{E}_0}{2}-\frac{c^2}{\gamma-1}\big)}\right)^{-1}
-\frac{2\gamma}{\gamma-1}<-\frac{1}{2},
\end{equation}
where the constant $$\widehat{E}_0:=\frac{u_0^2}{2}+\frac{c_0^2}{\gamma-1}.$$
We then define the following constants:
\begin{equation}\label{1972605}
\begin{aligned}
&M_1:=2\overline{M}_1,\quad
m_1:=\frac{\overline{m}_1}{2},\quad A_M=\frac{A_0}{2}+\frac{\pi}{4},\quad c_M:=2c_0,\quad
s_m:=\frac{s_0}{2},\quad
s_M:=2s_0,\\[4pt]
& q_m:=\frac{u_0}{2},\quad
\nu_1:=\frac{\gamma+1}{(\gamma-1)\cos^2 A_M}+1,\quad m_2:=\frac{m_1}{c_{M}^{\nu_1}},\quad
m:=\frac{1}{2}m_2 c_m^{\nu_1-\frac{2\gamma}{\gamma-1}},\\[4pt]&
\mathcal{B}:=1+\max\Bigg\{\frac{\gamma^{\frac{1}{\gamma-1}}s_m^{\frac{2-\gamma}{\gamma-1}}c_M^2 }{\gamma(\gamma-1)},~ \frac{\gamma^{\frac{1}{\gamma-1}}s_M^{\frac{2-\gamma}{\gamma-1}}c_M^2 }{\gamma(\gamma-1)}\Bigg\}, \\[6pt]&
\widehat{\mathcal{H}}:=\max_{\substack{A\in[0, A_{M}]\\  i=1,2; j=1,\cdot\cdot\cdot, 6}}\Big\{|\widehat{H_{ij}}|\Big\},~ \widetilde{\mathcal{H}}:=\max_{\substack{A\in[0, A_M]\\  i=1,2; j=1,\cdot\cdot\cdot, 6}}\Big\{|\widetilde{H_{ij}}|\Big\},
\\[4pt]&
\mathcal{N}:=\max\left\{
\frac{1}{q_m^2(\gamma s_m)^{\frac{1}{\gamma-1}}},~
\frac{1}{\gamma(\gamma-1)s_m},~
\frac{1}{\gamma^2(\gamma-1)^2s_m^2},~ \frac{1}{\gamma(\gamma-1)s_mq_m^2(\gamma s_m)^{\frac{1}{\gamma-1}}}\right\},
\\[8pt]&
 \widehat{\mathcal{M}}:=4\mathcal{B}\mathcal{N}(\widehat{\mathcal{H}}+1),\quad
 \widetilde{\mathcal{M}}:=4\mathcal{B}\mathcal{N}(\widetilde{\mathcal{H}}+1),
\\[4pt]&
\varepsilon_0:=\min\left\{m\gamma s_m,~\frac{m}{2(\widehat{\mathcal{M}}+1)},~ \frac{(\kappa-1)m}{2\widehat{\mathcal{M}}}, ~\frac{m}{3(\widetilde{\mathcal{M}}+1)}\right\}.
\end{aligned}
\end{equation}

\subsection{Flow in domains $\Sigma_{0}^{+}$ and $\Sigma_{0}^{-}$}
We are going to construct a global continuous and piecewise smooth supersonic solution to the problem (\ref{PSEU}), (\ref{bd1}).

 Through $B$ draw a forward $C_{+}$ characteristic curve $y=y_{+}^{B}(x)$ which can be determined by
$$\frac{{\rm d}y_{+}^{B}}{{\rm d}x}=\lambda_{+}\big(u_{in}(y_{+}^{B}), 0, c_{in}(y_{+}^{B})\big), \quad y_{+}^{B}(0)=-f(0).$$
 Through $D$ draw a forward $C_{-}$ characteristic curve  $y=y_{-}^{D}(x)$ which can be determined by
 $$\frac{{\rm d}y_{-}^{D}}{{\rm d}x}=\lambda_{-}\big(u_{in}(y_{-}^{D}), 0, c_{in}(y_{-}^{D})\big), \quad y_{-}^{D}(0)=f(0).$$
When $\epsilon$ is small the two curves intersect at some point $P=(x_{P}, 0)$ on the $x-$axis. It is easy to see that the flow in the region bounded by $\widehat{BP}$, $\widehat{DP}$, and $x=0$ is $(u, v, \rho, s)(x, y) =(u_{in}, v_{in}, \rho_{in}, s_{in})(y)$. Moreover, we have
\begin{equation}\label{72604}
|\delta_1(x, y)|=\Big|\frac{u_{in}'(y)}{\rho_{in}(y)}\Big|\leq \varepsilon\quad \mbox{and} \quad  |\delta_2(x, y)|=\Bigg|\frac{\cos\alpha_{in}(y) s_{in}'(y)}{c_{in}^{\small\frac{\gamma+1}{\gamma-1}}(y)}\Bigg|\leq\varepsilon\quad \mbox{on}\quad \widehat{BP}\cup\widehat{DP},
\end{equation}
where $\alpha_{in}(y)=\arctan \Big(\lambda_{+}\big(u_{in}(y), 0, c_{in}(y))\Big)$ and the constant
\begin{equation}\label{198101}
\varepsilon~:=~\max\left\{\sup\limits_{-f(0)\leq y\leq f(0)}\Big|\frac{u_{in}'(y)}{\rho_{in}(y)}\Big|,\quad
\sup\limits_{-f(0)\leq y\leq f(0)}\Bigg|\frac{s_{in}'(y)}{c_{in}^{\frac{\gamma+1}{\gamma-1}}(y)}\Bigg| \right\}.
\end{equation}

We next consider system (\ref{PSEU}) with the boundary data
\begin{equation}\label{72601}
 \left\{
   \begin{array}{ll}
     (u, v, \rho, s)=(u_{in}, v_{in}, \rho_{in}, s_{in})(y) & \hbox{on~~$\widehat{BP}$;} \\[4pt]
   (u, v)\cdot {\bf n_w}=0 & \hbox{on~~$W_{-}$.}
   \end{array}
 \right.
\end{equation}

By the classical results about the stability of classical solutions for quasilinear hyperbolic system, we can get the following lemma.
\begin{lem}
If $\epsilon$ is sufficiently small then the slip boundary problem (\ref{PSEU}), (\ref{72601}) admits a $C^1$ solution in a region $\Sigma_{0}^{-}$ bounded by $\widehat{BP}$, $W_{-}$, and $\widehat{PB_0}$, where $\widehat{PB_0}$ is a $C_{-}$ characteristic curve which passes through $P$ and ends up at a point $B_0$ on $W_{-}$; see Figure \ref{Fig2}. Moreover, the solution satisfies
\begin{equation}\label{72605}
\begin{aligned}
&|Ds(P)|=0, \quad |D\hat{E}(P)|=0, \quad \inf\limits_{\widehat{PB_0}}\bar{\partial}_{-}{c}\geq -M_1,\quad \sup\limits_{\widehat{PB_0}}\bar{\partial}_{-}{c}\leq -m_1,\\&\quad \sup\limits_{\widehat{PB_0}}A\leq A_{M}, \quad \inf\limits_{\widehat{PB_0}}c>\frac{\bar{c}_m}{2},\quad \inf\limits_{\widehat{PB_0}}\widehat{E}\geq \frac{\widehat{E}_0}{2},\quad \inf\limits_{\widehat{PB_0}}q\geq q_m,\\[4pt]&
\quad \sup\limits_{\widehat{PB_0}}c<c_{M}, \quad \sup\limits_{\widehat{PB_0}}|\delta_1|<\mathcal{B}\varepsilon, \quad \mbox{and}\quad \sup\limits_{\widehat{PB_0}}|\delta_2|\leq \varepsilon.
\end{aligned}
\end{equation}
\end{lem}
\begin{proof}
By computation, we have
\begin{equation}\label{72606}
|\bar{\partial}_{+}c|=|\sin\alpha_{in}(y) c_{in}'(y)|\leq \epsilon \quad \mbox{on}\quad \widehat{BP}.
\end{equation}
From (\ref{42201})
 we have
\begin{equation}\label{4203}
\begin{aligned}
R_{-}-R_{+}&=(\gamma-1)q \bar{\partial}_{0}\sigma
=-\frac{(\gamma-1)q f''(x)}{\big(1+[f'(x)]^2\big)^{\frac{3}{2}}}\quad \mbox{on}\quad W_{-}.
\end{aligned}
\end{equation}

If $\epsilon$ is small then the solution to the problem (\ref{PSEU}), (\ref{72601}) in $\Sigma_{0}^{-}$ is actually a small perturbation of the simple wave $\it S_{+}$ constructed in Section 3.1.
The desired estimates (\ref{72605}) can be obtained by using the boundary conditions (\ref{72604}), (\ref{72606}), (\ref{4203}) and by integrating $\bar{\partial}_{0}s=0$, (\ref{4401}), (\ref{4105}), (\ref{72803}), and (\ref{cd}) along characteristic curves. We omit the details.
\end{proof}

By symmetry, we can get a flow in a region $\Sigma_{0}^{+}$ bounded by $\widehat{DP}$, $W_{+}$, and $\widehat{PD_0}$, where $\widehat{PD_0}$ is a $C_{+}$ characteristic curve which passes through $P$ and ends up at some point $D_0$ on $W_{+}$. Moreover, we have
\begin{equation}\label{1972602}
\begin{aligned}
&|Ds(P)|=0, \quad |D\hat{E}(P)|=0, \quad \inf\limits_{\widehat{PD_0}}\bar{\partial}_{+}{c}\geq -M_1,\quad \sup\limits_{\widehat{PD_0}}\bar{\partial}_{+}{c}\leq -m_1,\\&\quad \sup\limits_{\widehat{PD_0}}A\leq A_{M},\quad \inf\limits_{\widehat{PD_0}}c>\frac{1}{2}\bar{c}_m,\quad\inf\limits_{\widehat{PD_0}}\widehat{E}\geq \frac{1}{2}\widehat{E}_0,\quad \inf\limits_{\widehat{PD_0}}q\geq q_m,\\[4pt]&\quad
 \sup\limits_{\widehat{PD_0}}c<c_{M}, \quad \sup\limits_{\widehat{PD_0}}|\delta_1|<\mathcal{B}\varepsilon, \quad \mbox{and}\quad \sup\limits_{\widehat{PD_0}}|\delta_2|\leq \varepsilon.
\end{aligned}
\end{equation}

\begin{rem}\label{2032001}
See  Figure \ref{Fig2}. Although  $(\bar{\partial}_{+}u, \bar{\partial}_{+}v, \bar{\partial}_{+}c)$  is discontinuous across the $C_{-}$ characteristic curve $\widehat{DP}$, $(\bar{\partial}_{-}u, \bar{\partial}_{-}v, \bar{\partial}_{-}c)$ is  continuous across $\widehat{DP}$. So, by the second equation of (\ref{form}) we know that $\omega$ is continuous across $\widehat{DP}$, since (\ref{form}) holds on both sides of $\widehat{DP}$. Since $\bar{\partial}_{-}s$ is continuous across $\widehat{DP}$, by $\bar{\partial}_{-}s+\bar{\partial}_{+}s=0$ we know that $\bar{\partial}_{+}s$ is also continuous across $\widehat{DP}$.
Similarly, we know that $\omega$ and $\bar{\partial}_{\pm}s$ are continuous across $\widehat{BP}$.
In view of this fact, one can see that  $\omega$ and $\bar{\partial}_{\pm}s$ are actually continuous.
\end{rem}

\subsection{Flow in domain $\Sigma_{1}$}
The purpose of this part is to construct the solution in domain $\Sigma_{1}$, as shown in Figure \ref{Fig2}.
For this purpose, we consider system (\ref{PSEU}) with the boundary data
\begin{equation}\label{1972601}
(u, v, \rho, s)=\left\{
                  \begin{array}{ll}
                    \big(u_{_{\widehat{PB_0}}}, v_{_{\widehat{PB_0}}}, \rho_{_{\widehat{PB_0}}}, s_{_{\widehat{PB_0}}}\big)(x, y) & \hbox{on $\widehat{PB_0}$;} \\[4pt]
                    \big(u_{_{\widehat{PD_0}}}, v_{_{\widehat{PD_0}}}, \rho_{_{\widehat{PD_0}}}, s_{_{\widehat{PD_0}}}\big)(x, y) & \hbox{on $\widehat{PD_0}$,}
                  \end{array}
                \right.
\end{equation}
where $(u_{_{\widehat{PB_0}}}, v_{_{\widehat{PB_0}}}, \rho_{_{\widehat{PB_0}}}, s_{_{\widehat{PB_0}}})(x, y)$ and $(u_{_{\widehat{PD_0}}}, v_{_{\widehat{PD_0}}}, \rho_{_{\widehat{PD_0}}}, s_{_{\widehat{PD_0}}})(x, y)$ represent the state on the characteristic curves $\widehat{PB_0}$ and $\widehat{PD_0}$, respectively. The characteristic curves $\widehat{PB_0}$ and $\widehat{PD_0}$ and the data on them are obtained in the last subsection.

Problem (\ref{PSEU}), (\ref{1972601}) is a Goursat problem.
 By (\ref{72605}) and (\ref{1972602}), we have
$
\bar{\partial}_{+}\widehat{E}_{{\widehat{PD_0}}}(P)=\bar{\partial}_{-}\widehat{E}_{{\widehat{PB_0}}}(P)=0$ and
$\bar{\partial}_{+}s_{_{\widehat{PD_0}}}(P)=\bar{\partial}_{-}s_{_{\widehat{PB_0}}}(P)=0.
$
Hence, we can check that compatibility conditions are satisfied at $P$.
So, existence of a local $C^1$ solution is known by the method of characteristics, see for example \cite{Li-Yu}.
In order to extend the local solution to global solution, we need to establish the a priori $C^1$ norm estimate of the solution.
In what follows,
we first assume that the Goursat problem (\ref{PSEU}), (\ref{1972601}) admits a $C^1$ local solution in some region $\Sigma_{1, loc}$, and then establish $C^1$ norm estimate of the solution on $\Sigma_{1, loc}$.
Since the local solution is constructed by the method of characteristics, through any point $F$ in $\Sigma_{1, loc}$ we can draw a backward $C_{-}$ ($C_{+}$, resp.) characteristic curve up to some point $F_{-}$ ($F_{+}$, resp.) on $\widehat{PD_0}$ ($\widehat{PB_0}$, resp.), and the two backward characteristic curves do not interact with each other as they go back toward to $\widehat{PD_0}$ and $\widehat{PB_0}$. Let $\Lambda_{1, F}$ be a closed domain bounded by characteristic curves $\widehat{F_{-}F}$, $\widehat{F_{+}F}$, $\widehat{PF_{+}}$,  and $\widehat{PF_{-}}$. We have that $\Lambda_{1, F}$ belongs to $\Sigma_{1, loc}$, as indicated in Figure \ref{Fig4}.

\begin{lem}\label{lem33}
If $\varepsilon<\varepsilon_0$ then the classical solution of the Goursat problem (\ref{PSEU}), (\ref{1972601}) satisfies
\begin{equation}\label{4103}
\begin{aligned}
c>0, \quad\frac{R_{+}}{c^{\frac{2\gamma}{\gamma-1}}}~<~-m,\quad \mbox{and}\quad\frac{R_{-}}{c^{\frac{2\gamma}{\gamma-1}}}~<~-m.
\end{aligned}
\end{equation}
\end{lem}
\begin{proof}
The proof of this lemma proceeds in four steps.

{\it Step 1.}
In this step, we shall prove $\frac{R_{\pm}}{c^{\frac{2\gamma}{\gamma-1}}}~<~-m$ on $\widehat{PB_0}\cup\widehat{PD_0}$.

By a direct computation, we have that if $c\geq c_m$ and $\frac{R_{\pm}}{c^{\nu_1}}<-\frac{m_2}{2}$
then
$$
\frac{R_{\pm}}{c^{\frac{2\gamma}{\gamma-1}}}~=~c^{\nu_1-\frac{2\gamma}{\gamma-1}}\frac{R_{\pm}}{c^{\nu_1}}
~< ~c_m^{\nu_1-\frac{2\gamma}{\gamma-1}}\frac{R_{\pm}}{c^{\nu_1}}~<~-\frac{1}{2}m_2 c_m^{\nu_1-\frac{2\gamma}{\gamma-1}}~=~- m.$$
Thus,  in view of $\inf\limits_{\widehat{PD_0}\cup \widehat{PB_0}}c>\frac{\bar{c}_m}{2}\geq c_m$, it suffices to prove  $\frac{R_{\pm}}{c^{\nu_1}}<-\frac{m_2}{2}$ on $\widehat{PB_0}\cup\widehat{PD_0}$.

In view of (\ref{1972605}), (\ref{1972602}), and $\frac{2\gamma}{\gamma-1}<\nu_1$, we have that
if $\varepsilon<\varepsilon_0$ then
\begin{equation}\label{41801}
\begin{aligned}
\frac{R_{+}}{c^{\nu_1}}=\frac{\bar{\partial}_{+}c}{c^{\nu_1}}-\frac{\delta_2\frac{c^{\frac{2\gamma}{\gamma-1}}}{\gamma(\gamma-1)s}}{\kappa c^{\nu_1}}&<-m_2+
\frac{\varepsilon c^{\frac{2\gamma}{\gamma-1}-\nu_1}}{2\gamma s}<-m_2+
\frac{\varepsilon c_m^{\frac{2\gamma}{\gamma-1}-\nu_1}}{2\gamma s_m}<-\frac{m_2}{2}\quad \mbox{on}\quad \widehat{PD_0}.
\end{aligned}
\end{equation}
By symmetry, we  have
\begin{equation}\label{1972604}
\frac{R_{-}}{c^{\nu_1}}~<~-\frac{m_2}{2}\quad \mbox{on}  \quad\widehat{PB_0}.
\end{equation}

From (\ref{1972604}) we have $\big(\frac{R_{-}}{c^{\nu_1}}\big)(P)<-\frac{m_2}{2}$.
Suppose that there exists a ``first" point on $\widehat{PD_0}$, such that $\big(\frac{R_{-}}{c^{\nu_1}}\big)=-\frac{m_2}{2}$ at this point. Then,
by the second equation of (\ref{cd10}), (\ref{1972605}), (\ref{1972602}), and (\ref{41801})  we have
\begin{equation}\label{1972606}
\begin{aligned}
  &c\bar{\partial}_{+}\Big(\frac{R_{-}}{c^{\nu_1}}\Big)\\
~<~&
  \frac{1}{c^{\nu_1}}\Bigg\{ \frac{(\gamma+1)R_{+}R_{-}}{2(\gamma-1)\cos^2A}+\frac{(\gamma+1) R_{+}R_{-}}{2(\gamma-1)\cos^2 A}
-\nu_1 R_{+}R_{-}\\&
\qquad\qquad+(\widehat{H}_{21}R_{-}+\widehat{H}_{22}R_{+})\frac{\delta_1 c^{\frac{2\gamma}{\gamma-1}}}{q^2(\gamma s)^{\frac{1}{\gamma-1}}}+(\widehat{H}_{23}R_{-}+\widehat{H}_{24}R_{+})
\frac{\delta_2 c^{\frac{2\gamma}{\gamma-1}}}{\gamma(\gamma-1)s}\\
&\qquad\qquad +\widehat{H}_{25}
\frac{\delta_2^2 c^{\frac{4\gamma}{\gamma-1}}}{\gamma^2(\gamma-1)^2s^2}+ \widehat{H}_{26}\frac{\delta_1
\delta_2 c^{\frac{4\gamma}{\gamma-1}}}{\gamma(\gamma-1)sq^2(\gamma s)^{\frac{1}{\gamma-1}}}\Bigg\}
\\~<~&\frac{1}{c^{\nu_1}}\Bigg\{-R_{+}R_{-}+(\widehat{H}_{21}R_{-}+\widehat{H}_{22}R_{+})
\frac{\delta_1 c^{\frac{2\gamma}{\gamma-1}}}{q^2(\gamma s)^{\frac{1}{\gamma-1}}}
+(\widehat{H}_{23}R_{-}+\widehat{H}_{24}R_{+})
\frac{\delta_2 c^{\frac{2\gamma}{\gamma-1}}}{\gamma(\gamma-1)s}\\
&\qquad\qquad +\widehat{H}_{25}
\frac{\delta_2^2 c^{\frac{4\gamma}{\gamma-1}}}{\gamma^2(\gamma-1)^2s^2}+ \widehat{H}_{26}\frac{\delta_1
\delta_2 c^{\frac{4\gamma}{\gamma-1}}}{\gamma(\gamma-1)sq^2(\gamma s)^{\frac{1}{\gamma-1}}}\Bigg\}\\
~<~&\frac{1}{c^{\nu_1}}\Bigg\{-R_{+}R_{-}-\big(|\widehat{H}_{21}+|\widehat{H}_{22}|\big)R_{+}
\frac{\delta_1 c^{\frac{2\gamma}{\gamma-1}}}{q^2(\gamma s)^{\frac{1}{\gamma-1}}}
-\big(|\widehat{H}_{23}|+|\widehat{H}_{24}|\big)R_{+}
\frac{\delta_2 c^{\frac{2\gamma}{\gamma-1}}}{\gamma(\gamma-1)s}\\
&\qquad\qquad +|\widehat{H}_{25}|
\frac{\delta_2^2 c^{\frac{4\gamma}{\gamma-1}}}{\gamma^2(\gamma-1)^2s^2}+ |\widehat{H}_{26}|\frac{\delta_1
\delta_2 c^{\frac{4\gamma}{\gamma-1}}}{\gamma(\gamma-1)sq^2(\gamma s)^{\frac{1}{\gamma-1}}}\Bigg\}
\\~<~&
c^{\nu_1}\left\{-\frac{R_{+}R_{-}}{c^{2\nu_1}}-\widehat{\mathcal{M}} \varepsilon c^{\frac{2\gamma}{\gamma-1}-\nu_1}\frac{R_{+}}{c^{\nu_1}}+\widehat{\mathcal{M}}\varepsilon^2c^
{2(\frac{2\gamma}{\gamma-1}-\nu_1)}\right\}
\\~<~&
c^{\nu_1}\Bigg\{\underbrace{\frac{R_{+}}{c^{\nu_1}}\Big(\frac{m_2}{4}-\widehat{\mathcal{M}} \varepsilon c_m^{\frac{2\gamma}{\gamma-1}-\nu_1}\Big)}_{<0}+\underbrace{\widehat{\mathcal{M}}\varepsilon^2c_m^{2(\frac{2\gamma}
{\gamma-1}-\nu_1)}-\frac{m_2^2}{8}}_{<0}\Bigg\}~<~0
\end{aligned}
\end{equation}
at this point, which leads to a contradiction. Thus, by an argument of continuity we have
$
\frac{R_{-}}{c^{\nu_1}}~<~-\frac{m_2}{2}$ on $\widehat{PD_0}$.
 Similarly, we have
$\frac{R_{+}}{c^{\nu_1}}~<~-\frac{m_2}{2}$ on $\widehat{PB_0}$.





\begin{figure}[htbp]
\begin{center}
\includegraphics[scale=0.4]{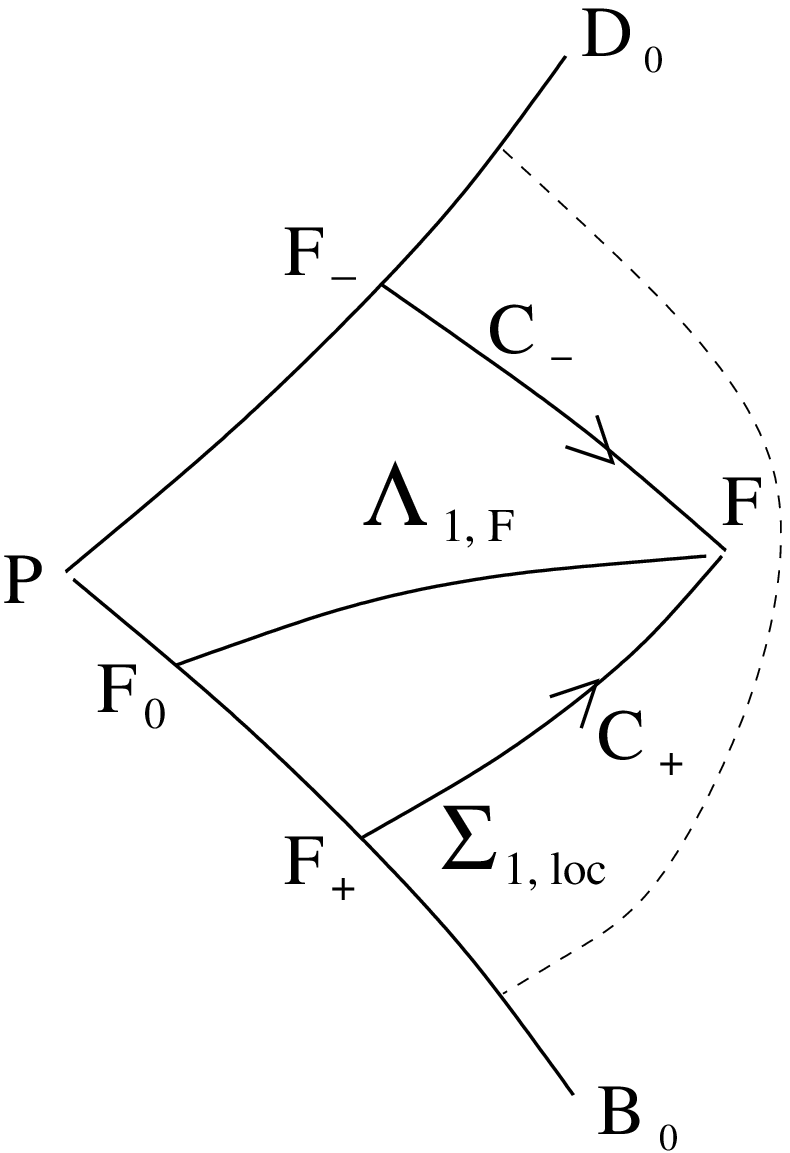}
\caption{ \footnotesize Domains $\Sigma_{1, loc}$ and $\Lambda_{1, F}$.}
\label{Fig4}
\end{center}
\end{figure}

{\it Step 2.}
In this step,
 we shall prove that for any $F\in\Sigma_{1, loc}$, if $\frac{R_{\pm}}{c^{\nu_1}}~<~-\frac{m_2}{2}$ and $c\geq c_m$ in $\Lambda_{1, F}\setminus\{F\}$, then $(\frac{R_{\pm}}{c^{\nu_1}})(F)~<~-\frac{m_2}{2}$.

In view of Remark \ref{2032001}, we can see that $\delta_1$ and $\delta_2$ are continuous across $\widehat{PD_0}$ and $\widehat{PB_0}$.
Then by (\ref{4105}), (\ref{72605}), and (\ref{1972602}) we have
\begin{equation}\label{42002}
|\delta_2|\leq \varepsilon \quad\mbox{on}\quad \Lambda_{1, F}.
\end{equation}

If $\frac{R_{\pm}}{c^{\nu_1}}~<~-\frac{m_2}{2}$ and $c\geq c_m$ in $\Lambda_{1, F}\setminus\{F\}$, then
\begin{equation}\label{41201}
\frac{\bar{\partial}_{\pm}c}{c^{\nu_1}}=\frac{R_{\pm}}{c^{\nu_1}}+
\frac{j\bar{\partial}_{\pm}s}{\kappa c^{\nu_1}}<-\frac{m_2}{2}+
\frac{\varepsilon c^{\frac{2\gamma}{\gamma-1}-\nu_1}}{2\gamma s}<-\frac{m_2}{2}+
\frac{\varepsilon c_m^{\frac{2\gamma}{\gamma-1}-\nu_1}}{2\gamma s_m}<-\frac{m_2}{4} \quad\mbox{on}\quad \Lambda_{1, F}\setminus\{F\}.
\end{equation}
Hence, by (\ref{4401}) and (\ref{32403}) we have
\begin{equation}\label{41901}
\bar{\partial}_{0}c<0\quad\mbox{and}\quad \bar{\partial}_{0}A=\frac{\sin ^3 A}{c^2 \cos A}\Big(\frac{c\bar{\partial}_{0}c}{\sin ^2 A}+\kappa c\bar{\partial}_{0}c \Big)<0\quad\mbox{on}\quad \Lambda_{1, F}\setminus\{F\}.
\end{equation}
Consequently, by the boundary data estimates (\ref{72605}) and (\ref{1972602}) we have
\begin{equation}\label{41202}
c<c_{M},\quad q\geq q_{m},\quad \mbox{and}\quad A\leq A_M\quad\mbox{on}\quad \Lambda_{1, F}.
\end{equation}

Through $F$ one can draw a backward $C_{0}$ characteristic curve, and this curve can intersect with $\widehat{DP}\cup\widehat{BP}$  at some point $F_{0}$.
Integrating (\ref{72803}) along this $C_0$ characteristic curve from $F_0$ to $F$ and using (\ref{72604}), (\ref{42002}), (\ref{41202}), and $\bar{\partial}_{0}c<0$, we get
$$
\Big|\frac{\omega}{\rho}(F)-\frac{\omega}{\rho}(F_0)\Big|<\varepsilon\max\Bigg\{\frac{\gamma^{\frac{1}{\gamma-1}}s_m^{\frac{2-\gamma}{\gamma-1}}c_M^2 }{\gamma(\gamma-1)}, \quad \frac{\gamma^{\frac{1}{\gamma-1}}s_M^{\frac{2-\gamma}{\gamma-1}}c_M^2 }{\gamma(\gamma-1)}\Bigg\},
$$
since $\delta_1$ is continuous across $\widehat{PD_0}$ and $\widehat{PB_0}$.
Hence,
\begin{equation}\label{41203}
\big|\delta_1(F)\big|<\mathcal{B}\varepsilon.
\end{equation}

If $(\frac{R_{-}}{c^{\nu_1}})(F)~=~-\frac{m_2}{2}$ and $(\frac{R_{+}}{c^{\nu_1}})(F)~\leq~-\frac{m_2}{2}$, then
by the second equation of (\ref{cd10}), (\ref{42002}), (\ref{41202}), and (\ref{41203}) we have
$$
\begin{aligned}
  c\bar{\partial}_{+}\Big(\frac{R_{-}}{c^{\nu_1}}\Big)~<~0
\quad\mbox{at}\quad F,
\end{aligned}
$$
as shown in (\ref{1972606}).
This leads to a contradiction, since $\frac{R_{+}}{c^{\nu_1}}~<~-\frac{m_2}{2}$ along $\widehat{F_{-}F}$.
Similarly, if $(\frac{R_{+}}{c^{\nu_1}})(F)~=~-\frac{m_2}{2}$ and $(\frac{R_{-}}{c^{\nu_1}})(F)~\leq~-\frac{m_2}{2}$, then
by the fourth equation of (\ref{cd10}), (\ref{41202}), and (\ref{41203}) we have
$
\bar{\partial}_{-}\big(\frac{R_{+}}{c^{\nu_1}}\big)<0
$
at $F$,
which leads to a contradiction. Thus, we have $\big(\frac{R_{\pm}}{c^{\nu_1}}\big)(F)~<~-\frac{m_2}{2}$.

\vskip 4pt

{\it Step 3.}
In this step, we shall prove that for any $F\in\Sigma_{1, loc}$, if $\frac{R_{\pm}}{c^{\frac{2\gamma}{\gamma-1}}}~<~-m$ and $c>0$ in $\Lambda_{1, F}\setminus\{F\}$ and $0<c(F)< c_m$, then $\big(\frac{R_{\pm}}{c^{\frac{2\gamma}{\gamma-1}}}\big)(F)~<~-m$.

When $\frac{R_{\pm}}{c^{\frac{2\gamma}{\gamma-1}}}~<~-m$ and $c>0$ in $\Lambda_{1, F}\setminus\{F\}$ we have
$$
\frac{\bar{\partial}_{\pm}c}{c^{\frac{2\gamma}{\gamma-1}}}=\frac{R_{\pm}}{c^{\frac{2\gamma}{\gamma-1}}}+
\frac{j\bar{\partial}_{\pm}s}{\kappa c^{\frac{2\gamma}{\gamma-1}}}<-m+
\frac{\varepsilon }{2\gamma s_m}<-\frac{m}{2} \quad\mbox{on}\quad \Lambda_{1, F}\setminus\{F\}.
$$
Using this we can get
$q(F)\geq q_m$, $A(F)\leq A_m$, $\big|\delta_2(F)\big|\leq \varepsilon$, and $\big|\delta_1(F)\big|<\mathcal{B}\varepsilon$, as shown in the previous step.

By (\ref{72602}), (\ref{72605}), and (\ref{1972602}) we have
\begin{equation}\label{1972607}
\begin{aligned}
\frac{\gamma+1}{(\gamma-1)\cos^2 A}-\frac{2\gamma}{\gamma-1}~=~&
\frac{\gamma+1}{(\gamma-1)}\Bigg(1-\frac{c^2}{2(\hat{E}-\frac{c^2}{\gamma-1})}\Bigg)^{-1}-\frac{2\gamma}{\gamma-1}\\~<~&
\frac{\gamma+1}{(\gamma-1)}\Bigg(1-\frac{c^2}{2(\frac{\widehat{E}_0}{2}-\frac{c^2}{\gamma-1})}\Bigg)^{-1}-\frac{2\gamma}{\gamma-1}~<~-\frac{1}{2}\quad \mbox{at}\quad F.
\end{aligned}
\end{equation}
Therefore,
if $\big(\frac{R_{+}}{c^{\frac{2\gamma}{\gamma-1}}}\big)(F)~=~-m$ and $\big(\frac{R_{-}}{c^{\frac{2\gamma}{\gamma-1}}}\big)(F)~\leq ~-m$, then by (\ref{1972607}) and the first equation  of (\ref{cd10}) we have
\begin{equation}\label{42007}
\begin{aligned}
  &c\bar{\partial}_{-}\Big(\frac{R_{+}}{c^{\frac{2\gamma}{\gamma-1}}}\Big)\\<&
  \frac{1}{c^{\frac{2\gamma}{\gamma-1}}}\Bigg\{ \frac{(\gamma+1)R_{+}R_{-}}{2(\gamma-1)\cos^2A}+\frac{(\gamma+1)}{2(\gamma-1)\cos^2 A}
   R_{+}R_{-}-\frac{2\gamma}{\gamma-1} R_{+}R_{-}\\&
\qquad\qquad+(\widehat{H}_{11}R_{-}+\widehat{H}_{12}R_{+})
\frac{\delta_1 c^{\frac{2\gamma}{\gamma-1}}}{q^2(\gamma s)^{\frac{1}{\gamma-1}}}+(\widehat{H}_{13}R_{-}+\widehat{H}_{14}R_{+})
\frac{\delta_2 c^{\frac{2\gamma}{\gamma-1}}}{\gamma(\gamma-1)s}\\
&\qquad\qquad +\widehat{H}_{15}\frac{\delta_2 ^2c^{\frac{4\gamma}{\gamma-1}}}{\gamma^2(\gamma-1)^2s^2}+ \widehat{H}_{16}\frac{\delta_1\delta_2 c^{\frac{2\gamma}{\gamma-1}}}{q^2\gamma(\gamma-1)s(\gamma s)^{\frac{1}{\gamma-1}}}\Bigg\}
\\<&\frac{1}{c^{\frac{2\gamma}{\gamma-1}}}\Bigg\{-\frac{ R_{+}R_{-}}{2}+(\widehat{H}_{11}R_{-}+\widehat{H}_{12}R_{+})
\frac{\delta_1 c^{\frac{2\gamma}{\gamma-1}}}{q^2(\gamma s)^{\frac{1}{\gamma-1}}}+(\widehat{H}_{13}R_{-}+\widehat{H}_{14}R_{+})
\frac{\delta_2 c^{\frac{2\gamma}{\gamma-1}}}{\gamma(\gamma-1)s}\\
&\qquad\qquad +\widehat{H}_{15}\frac{\delta_2 ^2c^{\frac{4\gamma}{\gamma-1}}}{\gamma^2(\gamma-1)^2s^2}+ \widehat{H}_{16}\frac{\delta_1\delta_2 c^{\frac{2\gamma}{\gamma-1}}}{q^2\gamma(\gamma-1)s(\gamma s)^{\frac{1}{\gamma-1}}}\Bigg\}\\
<&\frac{1}{c^{\frac{2\gamma}{\gamma-1}}}\Bigg\{-\frac{ R_{+}R_{-}}{2}-\big(|\widehat{H}_{11}|+|\widehat{H}_{12}|\big)R_{-}
\frac{\delta_1 c^{\frac{2\gamma}{\gamma-1}}}{q^2(\gamma s)^{\frac{1}{\gamma-1}}}-\big(|\widehat{H}_{13}|+|\widehat{H}_{14}|\big)R_{-}
\frac{\delta_2 c^{\frac{2\gamma}{\gamma-1}}}{\gamma(\gamma-1)s}\\
&\qquad\qquad +|\widehat{H}_{15}|\frac{\delta_2 ^2c^{\frac{4\gamma}{\gamma-1}}}{\gamma^2(\gamma-1)^2s^2}+ |\widehat{H}_{16}|\frac{\delta_1\delta_2 c^{\frac{2\gamma}{\gamma-1}}}{q^2\gamma(\gamma-1)s(\gamma s)^{\frac{1}{\gamma-1}}}\Bigg\}
\\<&
c^{\frac{2\gamma}{\gamma-1}}\left\{-\frac{R_{+}R_{-}}{2c^{\frac{4\gamma}{\gamma-1}}}-\widehat{M} \varepsilon \frac{R_{-}}{c^{\frac{2\gamma}{\gamma-1}}}+\widehat{M}\varepsilon^2\right\}
\\<&
c^{\frac{2\gamma}{\gamma-1}}\left\{\frac{R_{-}}{c^{\frac{2\gamma}{\gamma-1}}}\Big(\frac{m}{2}-\widehat{M} \varepsilon \Big)+\widehat{M}\varepsilon^2-\frac{m^2}{4}\right\}~<~0
\qquad\mbox{at}\quad F.
\end{aligned}
\end{equation}
This leads to a contradiction, since $\frac{R_{+}}{c^{\frac{2\gamma}{\gamma-1}}}~<~-m$ along $\widehat{F_{-}F}$.

Similarly, If $\big(\frac{R_{-}}{c^{\frac{2\gamma}{\gamma-1}}}\big)(F)~=~-m$ and $\big(\frac{R_{+}}{c^{\frac{2\gamma}{\gamma-1}}}\big)(F)~\leq ~-m$, then by the second equation  of (\ref{cd6}) we have
$
\bar{\partial}_{-}\big(\frac{R_{-}}{c^{\frac{2\gamma}{\gamma-1}}}\big)<0
$
at $F$,
which leads to a contradiction.
Thus, we have $\big(\frac{R_{\pm}}{c^{\frac{2\gamma}{\gamma-1}}}\big)(F)~<~-m$.

{\it Step 4.}
In this step, we shall prove that for any $F\in\Sigma_{1, loc}$, if $c>0$ in $\Lambda_{1, F}\setminus\{F\}$ then $c(F)>0$.

By the results of the previous steps and an argument of continuity, we can get that if $c>0$ in $\Lambda_{1, F}\setminus\{F\}$, then $\big|\delta_1\big|\leq \mathcal{B}\varepsilon$, $\big|\delta_2\big|\leq \varepsilon$,
$\frac{R_{\pm}}{c^{\frac{2\gamma}{\gamma-1}}}~<~-m$ and
$\frac{\bar{\partial}_{\pm}c}{c^{\frac{2\gamma}{\gamma-1}}}<-\frac{m}{2}$
in $\Lambda_{1, F}\setminus\{F\}$.

Since $\varepsilon<\varepsilon_0\leq \frac{(\kappa-1)m}{2\widehat{\mathcal{M}}}$, there exists a $A_{s}>0$ such that if $A<A_s$ then $$(1+\kappa)\Big(\frac{\kappa-1}{\kappa+1}\cos^2 A-\sin^2A\Big)\frac{m}{2} -\frac{\widehat{\mathcal{M}}\varepsilon}{2}>0.$$

If $c(F)=0$
then by $A=\arcsin \frac{c}{q}$ and $q\geq q_m$ we have $A(F)=0$.
Then there exist a $\bar{F}_{+}$ on $\widehat{F_{+}F}$ and a $\bar{F}_{-}$  on $\widehat{F_{-}F}$ such that $A<A_s$ on $\widehat{\bar{F}_{+}F}$ and $\widehat{\bar{F}_{-}F}$.
Therefore, by (\ref{192307}) we have
\begin{equation}\label{41301}
\begin{aligned}
\bar{\partial}_{+}\alpha~=~&c^{\frac{\gamma+1}{\gamma-1}}\tan A\left\{-(1+\kappa)\Big(\frac{\kappa-1}{\kappa+1}\cos^2 A-\sin^2A\Big)\frac{\bar{\partial}_{+}c }{c^{\frac{2\gamma}{\gamma-1}}} -\frac{\delta_1}{q^2(\gamma s)^{\frac{1}{\gamma-1}}}
+\frac{\delta_2\cos 2A}{\gamma(\gamma-1)s }\right\}\\~>~&c^{\frac{\gamma+1}{\gamma-1}}\tan A\left\{(1+\kappa)\Big(\frac{\kappa-1}{\kappa+1}\cos^2 A-\sin^2A\Big)\frac{m}{2} -\frac{\mathcal{B}\varepsilon}{q_m^2(\gamma s_m)^{\frac{1}{\gamma-1}}}
-\frac{\varepsilon}{\gamma(\gamma-1)s_m }\right\}\\~>~&c^{\frac{\gamma+1}{\gamma-1}}\tan A\left\{(1+\kappa)\Big(\frac{\kappa-1}{\kappa+1}\cos^2 A-\sin^2A\Big)\frac{m}{2} -\frac{\widehat{\mathcal{M}}\varepsilon}{2}\right\}>0\quad \mbox{along}\quad \widehat{\bar{F}_{+}F}.
\end{aligned}
\end{equation}
Similarly, by (\ref{192309}) we have $c\bar{\partial}_{-}\beta<0$ along $\widehat{\bar{F}_{-}F}$.
Therefore, by $\alpha(F)-\beta(F)=2A(F)=0$ we know that the forward $C_0$ characteristic curve issuing from any point on $\widehat{\bar{F}_{+}F}$ and the forward $C_0$ characteristic curve issuing from any point on $\widehat{\bar{F}_{-}F}$ intersect at $F$ as illustrated in Figure \ref{Fig82}, which leads to a contradiction in view of the conservation of mass (see \cite{CQ1}, p. 2953).

\begin{figure}[htbp]
\begin{center}
\includegraphics[scale=0.55]{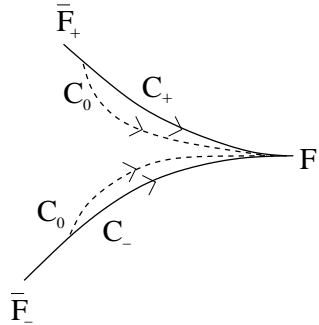}
\caption{ \footnotesize An impossible vacuum point.}
\label{Fig82}
\end{center}
\end{figure}

Combining with the results of the above four steps and using the method of continuity,
we can get this lemma.
\end{proof}

\begin{rem}\label{rem8102}
From the proof of Lemma \ref{lem33}, we can see that the classical solution of the Goursat problem (\ref{PSEU}) and (\ref{1972601}) satisfies
\begin{equation}\label{41903}
\begin{aligned}
&q\geq q_m,\quad \hat{E}>\frac{\widehat{E}_0}{2}, \quad  0<A\leq A_{M}, \quad \frac{\bar{\partial}_{\pm}c}{c^{\frac{2\gamma}{\gamma-1}}}<-\frac{m}{2}, \\&\quad \frac{R_{\pm}}{c^{\frac{2\gamma}{\gamma-1}}}<-m, \quad \big|\delta_1\big|\leq \mathcal{B}\varepsilon, \quad \mbox{and}\quad \big|\delta_2\big|\leq \varepsilon.
\end{aligned}
\end{equation}
\end{rem}

\begin{rem}\label{rem8101}
From steps 1 and 2 in the proof of Lemma \ref{lem33}, we can see that the classical solution of the Goursat problem (\ref{PSEU}) and (\ref{1972601}) also satisfy
\begin{equation}\label{198102}
\frac{R_{\pm}}{c^{\nu_1}}<-\frac{m_2}{2} \quad \mbox{as}\quad c\geq c_m.
\end{equation}
\end{rem}

\begin{lem}\label{34}
If $\varepsilon<\varepsilon_0$ then the classical solution of the Goursat problem (\ref{PSEU}), (\ref{1972601}) satisfies
\begin{equation}\label{4104}
R_{\pm}>-M_1-\frac{\varepsilon c_M^{\frac{2\gamma}{\gamma-1}}}{2\gamma s_m}.
\end{equation}
\end{lem}
\begin{proof}
It is easy to check by (\ref{72605}) and (\ref{1972602}) that
\begin{equation}\label{41101}
R_{+}\mid_{\widehat{PD_0}}~>~-M_1-\frac{\varepsilon c_M^{\frac{2\gamma}{\gamma-1}}}{2\gamma s_m}\quad
\mbox{and}\quad
R_{-}\mid_{\widehat{PB_0}}~>~-M_1-\frac{\varepsilon c_M^{\frac{2\gamma}{\gamma-1}}}{2\gamma s_m}.
\end{equation}

From the first equation of (\ref{cd8}), we have
$$
\begin{aligned}
c\bar{\partial}_{-}R_{+}~=~&
   \frac{(\gamma+1)R_{+}^2}{2(\gamma-1)\cos^2A}+\frac{(\gamma+1)-2\sin^2 2A}{2(\gamma-1)\cos^2 A}
   R_{+}R_{-}\\&
+(\widetilde{H}_{11}R_{-}+\widetilde{H}_{12}R_{+})
\frac{\delta_1 c^{\frac{2\gamma}{\gamma-1}}}{q^2(\gamma s)^{\frac{1}{\gamma-1}}}+(\widetilde{H}_{13}R_{-}+\widetilde{H}_{14}R_{+})
\frac{\delta_2 c^{\frac{2\gamma}{\gamma-1}}}{\gamma(\gamma-1)s}\\
& +\widetilde{H}_{15}\frac{\delta_2 ^2c^{\frac{4\gamma}{\gamma-1}}}{\gamma^2(\gamma-1)^2s^2}+ \widetilde{H}_{16}\frac{\delta_1\delta_2 c^{\frac{4\gamma}{\gamma-1}}}{\gamma(\gamma-1)s q^2(\gamma s)^{\frac{1}{\gamma-1}}}\\~>~&
\frac{1}{2\cos^2 A} R_{+}R_{-}
+(\widetilde{H}_{11}R_{-}+\widetilde{H}_{12}R_{+})
\frac{\delta_1 c^{\frac{2\gamma}{\gamma-1}}}{q^2(\gamma s)^{\frac{1}{\gamma-1}}}+(\widetilde{H}_{13}R_{-}+\widetilde{H}_{14}R_{+})
\frac{\delta_2 c^{\frac{2\gamma}{\gamma-1}}}{\gamma(\gamma-1)s}\\
& +\widetilde{H}_{15}\frac{\delta_2 ^2c^{\frac{4\gamma}{\gamma-1}}}{\gamma^2(\gamma-1)^2s^2}+ \widetilde{H}_{16}\frac{\delta_1\delta_2 c^{\frac{4\gamma}{\gamma-1}}}{\gamma(\gamma-1)s q^2(\gamma s)^{\frac{1}{\gamma-1}}}\\~>~&
\frac{ c^{\frac{4\gamma}{\gamma-1}}}{2\cos^2 A}\Bigg\{\frac{R_{+}R_{-}}{c^{\frac{4\gamma}{\gamma-1}}}+\widetilde{\mathcal{M}} \varepsilon \frac{R_{-}}{c^{\frac{2\gamma}{\gamma-1}}}+\widetilde{\mathcal{M}} \varepsilon \frac{R_{+}}{c^{\frac{2\gamma}{\gamma-1}}}-\widetilde{\mathcal{M}}\varepsilon^2\Bigg\}
\\~>~&\frac{ c^{\frac{4\gamma}{\gamma-1}}}{2\cos^2 A}\Bigg\{\Big(
\frac{R_{+}}{c^{\frac{2\gamma}{\gamma-1}}}+\frac{R_{-}}{c^{\frac{2\gamma}{\gamma-1}}}\Big)\Big(-\frac{m}{3}
+\widetilde{\mathcal{M}}\varepsilon\Big)
+\frac{m^2}{3}-\widetilde{\mathcal{M}}\varepsilon^2\Bigg\}>0,
\end{aligned}
$$
since $\varepsilon<\varepsilon_0\leq \frac{m}{3(\widetilde{\mathcal{M}}+1)}$.
Combining with this and (\ref{41101}), we have $R_{+}>-M_1-\frac{\varepsilon c_M^{\frac{2\gamma}{\gamma-1}}}{2\gamma s_m}$.

Similarly, by the second equation of (\ref{cd8}) we have
$
c\bar{\partial}_{+}R_{-}>0.
$
Combining this with (\ref{41101}) we get $R_{-}>-M_1-\frac{\varepsilon c_M^{\frac{2\gamma}{\gamma-1}}}{2\gamma s_m}$.

We then complete the proof of this lemma.
\end{proof}

Using (\ref{4401}), (\ref{192303})--(\ref{82505}), and Lemmas \ref{lem33} and \ref{34}, we can establish uniform a priori $C^1$ norm estimate of the solution.
Therefore, by the local existence result and the standard continuity extension method (cf. \cite{LiT}), one can extend the local solution to a whole determinate region of the Goursat problem. We then have the following lemma.
\begin{lem}
The Goursat problem (\ref{PSEU}), (\ref{1972601}) admits a global classical solution in a region $\Sigma_1$
  bounded by $\widehat{PD_0}$, $\widehat{PB_0}$, a forward $C_{-}$ characteristic curve $C_{-}^{D_0}$ issuing from $D_0$, and a forward $C_{+}$ characteristic curve  $C_{+}^{B_0}$ issuing from $B_0$.  Moreover, the solution satisfies
(\ref{41903}) and (\ref{198102}).
\end{lem}

\vskip 4pt

\subsection{Flow in $\Sigma_{1}^{+}$ and $\Sigma_{1}^{-}$}The purpose of this subsection is to construct the solution in domains $\Sigma_{1}^{+}$ and $\Sigma_{1}^{-}$.
For this purpose, we consider system (\ref{PSEU}) with the boundary conditions:
\begin{equation}\label{bd4}
\left\{
  \begin{array}{ll}
    (u, v, \rho, s)=\big(u_{{C_{+}^{B_0}}}, v_{{C_{+}^{B_0}}}, \rho_{{C_{+}^{B_0}}}, s_{{C_{+}^{B_0}}}\big)(x, y)& \hbox{on $ C_{+}^{B_0}$;}  \\[8pt]
  (u, v)\cdot {\bf n_w}=0& \hbox{on $W_{-}$,}
  \end{array}
\right.
\end{equation}
where $\big(u_{{C_{+}^{B_0}}}, v_{{C_{+}^{B_0}}}, \rho_{{C_{+}^{B_0}}}, s_{{C_{+}^{B_0}}}\big)(x, y)$ is the solution of the Goursat problem  (\ref{PSEU}), (\ref{1972601}) on $C_{+}^{B_0}$.

\begin{figure}[htbp]
\begin{center}
\includegraphics[scale=0.45]{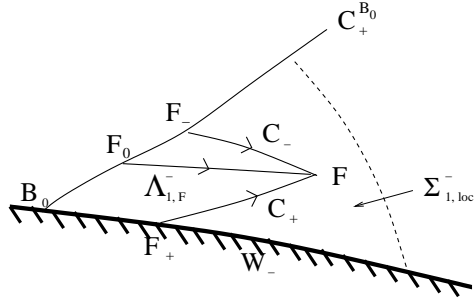}
\caption{ \footnotesize Domains $\Sigma_{1, loc}^{-}$ and ${\Lambda_{1, F}^{-}}$.}
\label{Fig5}
\end{center}
\end{figure}

Existence of a local $C^1$ solution is known by the method of characteristic, see \cite{CQ2, Li-Yu}.
In order to extend the local solution to global solution we need to establish the a priori $C^1$ norm estimate of the solution.
In what follows,
we assume that the boundary value problem (\ref{PSEU}), (\ref{bd4}) has a classical solution in some region $\Sigma_{1, loc}^{-}$, and then establish the $C^1$ norm estimate of the solution on $\Sigma_{1, loc}^{-}$. Since the local classical solution is constructed by the method of characteristics, through any point $F\in \Sigma_{1, loc}^{-}$ ($F$ can be also on $W_{-}$) one can draw a backward $C_{-}$ ($C_{+}$, resp.) characteristic curve up to a point $F_{-}$ ($F_{+}$, resp.) on $C_{+}^{B_0}$ ($W_{-}$, resp.), and the closed domain ${\Lambda_{1, F}^{-}}$ bounded by $\widehat{F_{-}F}$, $\widehat{F_{+}F}$, $\widehat{B_0F_{-}}$,  and $\widehat{B_0F_{+}}$ belongs to $\Sigma_{1, loc}^{-}$, as indicated in Figure \ref{Fig5}.


\begin{lem}\label{lem36}
If $\varepsilon<\varepsilon_0$ then the classical solution of the boundary value problem (\ref{PSEU}), (\ref{bd4}) satisfies $\frac{R_{\pm}}{c^{\frac{2\gamma}{\gamma-1}}}~<~-m$ in ${\Sigma_{1, loc}^{-}}$ and $c>0$ in ${\Sigma_{1, loc}^{-}}\setminus W_{-}$.
\end{lem}
\begin{proof}

The proof of this lemma proceeds in four steps.

{\it Step 1.} From  Remarks \ref{rem8101} and \ref{rem8102}, we have that along
$ C_{+}^{B_0}$,
\begin{equation}
\frac{R_{+}}{c^{\nu_1}}<-\frac{m_2}{2} \quad \mbox{as}\quad c\geq c_m\quad\mbox{and}\quad \frac{R_{+}}{c^{\frac{2\gamma}{\gamma-1}}}<-m \quad \mbox{as}\quad c>0.
\end{equation}
By (\ref{4203}) we have
$$
R_{-}(B_0)=R_{+}(B_0)-\frac{(\gamma-1)q f''(x)}{\big(1+[f'(x)]^2\big)^{\frac{3}{2}}}<R_{+}(B_0).
$$
Thus, we have
\begin{equation}\label{198103}
\Big(\frac{R_{-}}{c^{\nu_1}}\Big)(B_0)<-\frac{m_2}{2} \quad \mbox{as}\quad c(B_0)\geq c_m,\quad\mbox{and}\quad \Big(\frac{R_{-}}{c^{\frac{2\gamma}{\gamma-1}}}\Big)(B_0)<-m \quad \mbox{as}\quad c(B_0)>0.
\end{equation}

If there is a point on $ C_{+}^{B_0}$ such that $c\geq c_m$ and $\frac{R_{-}}{c^{\nu_1}}=-\frac{m_2}{2}$ at this point.
Then by the second equation of (\ref{cd10}) we have $
  c\bar{\partial}_{-}\big(\frac{R_{+}}{c^{\nu_1}}\big)~<~0
$, as shown in (\ref{1972606}). If there is a point on $ C_{+}^{B_0}$ such that $c<c_m$ and $\frac{R_{-}}{c^{\frac{2\gamma}{\gamma-1}}}=-m$ at this point.
Then by the second equation of (\ref{cd10}) we have $
  c\bar{\partial}_{-}\big(\frac{R_{+}}{c^{\frac{2\gamma}{\gamma-1}}}\big)~<~0
$, as shown in (\ref{42007}). Therefore, by (\ref{198103}), $\bar{\partial}_{\pm}c<0$, and an argument of continuity we have that along
$ C_{+}^{B_0}$,
\begin{equation}
\frac{R_{-}}{c^{\nu_1}}<-\frac{m_2}{2} \quad \mbox{as}\quad c\geq c_m\quad\mbox{and}\quad \frac{R_{-}}{c^{\frac{2\gamma}{\gamma-1}}}<-m \quad \mbox{as}\quad c>0.
\end{equation}
\vskip 4pt

\vskip 4pt
{\it Step 2.}
 In this step, we shall prove that for any $F\in\Sigma_{1, loc}^{-}$, if $\frac{R_{\pm}}{c^{\nu_1}}~<~-\frac{m_2}{2}$ and $c\geq c_m$ in $\Lambda_{1, F}^{-}\setminus\{F\}$, then $\frac{R_{\pm}}{c^{\nu_1}}~<~-\frac{m_2}{2}$ at $F$. 

Through any point in $\Sigma_{1, loc}^{-}$ one can draw a backward $C_{0}$ characteristic curve, and this curve can intersect with $\widehat{BP}$ at some point. By Remark \ref{2032001}, we can get that $\delta_1$ and $\delta_2$ are also continuous across $C_{+}^{B_0}$. Therefore, as shown in (\ref{42002})--(\ref{41203}),
we have 
$$
|\delta_2|\leq \varepsilon \quad \mbox{and}\quad \big|\delta_1\big|\leq \mathcal{B}\varepsilon\quad \mbox{in}\quad\Sigma_{1, loc}^{-}.
$$

If $\big(\frac{R_{+}}{c^{\nu_1}}\big)(F)~=~-\frac{m_2}{2}$ and $\big(\frac{R_{-}}{c^{\nu_1}}\big)(F)~\leq~-\frac{m_2}{2}$, then we get
$
  c\bar{\partial}_{-}\big(\frac{R_{+}}{c^{\nu_1}}\big)~<~0
$ at $F$. (The proof for this is the same as (\ref{1972606}), so we omit it.)
This leads to a contradiction, since $\frac{R_{+}}{c^{\nu_1}}~<~-\frac{m_2}{2}$ in $\Lambda_{1, F}^{-}\setminus\{F\}$.
If $\big(\frac{R_{-}}{c^{\nu_1}}\big)(F)~=~-\frac{m_2}{2}$ and $\big(\frac{R_{+}}{c^{\nu_1}}\big)(F)~<~-\frac{m_2}{2}$, then by (\ref{4203}) and $f''>0$ we know that $F$ does not lie on $W_{-}$, and hence $\widehat{F_{+}F}$ exists.
Thus, by the second equation of (\ref{cd6}) we have
$
\bar{\partial}_{+}\big(\frac{R_{-}}{c^{\nu_1}}\big)<0
$
 at $F$. (The proof for this is also the same as (\ref{1972606}), so we omit it.)
This leads to a contradiction, since $\frac{R_{+}}{c^{\nu_1}}~<~-\frac{m_2}{2}$ along $\widehat{F_{+}F}$.
Thus, we have
$\big(\frac{R_{\pm}}{c^{\nu_1}}\big)(F)~<~-\frac{m_2}{2}$.

\vskip 4pt
{\it Step 3.}
Using the method in the third step of the proof of Lemma \ref{lem33} and $f''>0$, one can get that if $\frac{R_{\pm}}{c^{\frac{2\gamma}{\gamma-1}}}~<~-m$ and $c>0$ in $\Lambda_{1, F}^{-}\setminus\{F\}$  and $0<c(F)< c_m$, then $\frac{R_{\pm}}{c^{\frac{2\gamma}{\gamma-1}}}(F)~<~-m$.

\vskip 4pt
{\it Step 4.}
As shown in the fourth step of the proof of Lemma \ref{lem33}, we can prove that
for any point $F\in {\Sigma_{1, loc}^{-}}\setminus W_{-}$,
if
$c>0$ in $\Lambda_{1, F}^{-}\setminus\{F\}$ then $c(F)>0$.

Therefore, by the method of continuity we can get this lemma.
\end{proof}

\begin{rem}
Actually, we also have that the classical solution  of the boundary value problem (\ref{PSEU}), (\ref{bd4}) satisfies
(\ref{41903}) and (\ref{198102}).
\end{rem}
\vskip 4pt

\begin{lem}\label{lem38}
Let $$M_2=\max\limits_{x\in [0, +\infty)}\left\{\frac{ f''(x)}{(1+[f'(x)]^2)^{\frac{3}{2}}}\right\}\quad \mbox{and} \quad q_{_M}=\left(u_{in}^2(f(0))+\frac{2c_{in}^2(f(0))}{\gamma-1}\right)^{\frac{1}{2}}. $$
Then when  $\varepsilon<\varepsilon_0$ the classical solution of the boundary value problem (\ref{PSEU}), (\ref{bd4}) satisfies
$$
R_{\pm}>-M_1-(\gamma-1)q_{_M}M_2-\frac{\varepsilon c_M^{\frac{2\gamma}{\gamma-1}}}{2\gamma s_m}.
$$
\end{lem}
\begin{proof}
From (\ref{41903}) we can get
\begin{equation}\label{4205}
\bar{\partial}_{-}R_{+}>0\quad\mbox{and}\quad\bar{\partial}_{+}R_{-}>0,
\end{equation}
as shown in the proof of Lemma \ref{34}.
From Lemma \ref{34} we have
$$
R_{+} ~>~-M_1-\frac{\varepsilon c_M^{\frac{2\gamma}{\gamma-1}}}{2\gamma s_m}\quad\mbox{on}\quad C_{+}^{B_0}.
$$
Thus,  we have
\begin{equation}\label{4204}
R_{+}>-M_1-\frac{\varepsilon c_M^{\frac{2\gamma}{\gamma-1}}}{2\gamma s_m}.
\end{equation}

From (\ref{4203}) and (\ref{4204}), we have
$$
\begin{aligned}
R_{-}~=~R_{+}-\frac{(\gamma-1)q f''(x)}{\big(1+[f'(x)]^2\big)^{\frac{3}{2}}}~>~ -M_1-(\gamma-1)q_{_M}M_2-\frac{\varepsilon c_M^{\frac{2\gamma}{\gamma-1}}}{2\gamma s_m}\quad \mbox{along}\quad W_{-},
\end{aligned}
$$
since $q\leq q_{_M}$ along $W_{-}$ by (\ref{4401}) and $\bar{\partial}_{0}c<0$.
Thus, by (\ref{4205}) we have
$$
R_{-}>-M_1-(\gamma-1)q_{_M}M_2-\frac{\varepsilon c_M^{\frac{2\gamma}{\gamma-1}}}{2\gamma s_m}.
$$
We then complete the proof of this lemma.
\end{proof}

Using (\ref{4401}), (\ref{192303})--(\ref{82505}), and Lemmas \ref{lem36} and \ref{lem38}, we can establish uniform a priori $C^1$ norm estimate of the solution.
Therefore, by the local existence result and the standard continuity extension method, we can extend the local solution to a whole determinate region of the slip boundary problem; cf. \cite{CQ2,LiT}. We then have the following lemma.
\begin{lem}
The boundary value problem (\ref{PSEU}), (\ref{bd4}) admits a gloabl $C^1$ solution in a region $\Sigma_1^{-}$ as shown in Figure \ref{Domain}.
Moreover, the solution satisfies
(\ref{41903})  and (\ref{198102}).
\end{lem}

By symmetry we can construct a solution in region $\Sigma_1^{+}$, as shown in Figure \ref{Domain}.

\subsection{Global solution of the boundary value problem  (\ref{PSEU}), (\ref{bd1})}
By repeatedly solving Goursat problems and slip boundary problems, we can get the solution
in regions $\Sigma_{0}$, $\Sigma_{0}^{+}$, $\Sigma_{0}^{-}$, $\Sigma_{1}$, $\Sigma_{1}^{+}$, $\Sigma_{1}^{-}$, $\Sigma_{2}$, $\Sigma_{2}^{+}$, $\Sigma_{2}^{-}$, $\cdot\cdot\cdot$, as illustrated in Figure \ref{Fig2}. Moreover, the solution satisfies
(\ref{41903}).
A natural question is whether the region $\Sigma$ can be covered by the determinant regions of these Goursat problems and slip boundary problems.
We are going to show that the flow in $\Sigma$ can be obtained after solving a finite number of Gourst problems and slip boundary problems.


By a direct computation, we have
$$
\sin^2 A=\frac{c^2}{2\Big(\hat{E}-\frac{c^2}{\gamma-1}\Big)}
<\frac{c^2}{2\big(\frac{\widehat{E}_0}{2}-\frac{c^2}{\gamma-1}\big)},
$$
since $\hat{E}>\frac{\widehat{E}_0}{2}$.
So, there exists a
$$0<c_g< \frac{q_m\arctan f'(x_{_{B_0}})}{9(\kappa-1)},$$ such that if $c<c_{g}$ then
\begin{equation}\label{41302}
(1+\kappa)\Big(\frac{\kappa-1}{\kappa+1}\cos^2 A-\sin^2A\Big)\frac{m}{2} -\frac{\widehat{\mathcal{M}}\varepsilon}{2}>0
\end{equation}
and
\begin{equation}\label{198301}
 A<\min\left\{\frac{\arctan f'(x_{_{B_0}})}{3},~ \frac{\pi}{3}\right\}.
\end{equation}
Here, $x_{_{B_0}}$ is the abscissa of the point $B_0$.

\begin{figure}[htbp]
\begin{center}
\includegraphics[scale=0.41]{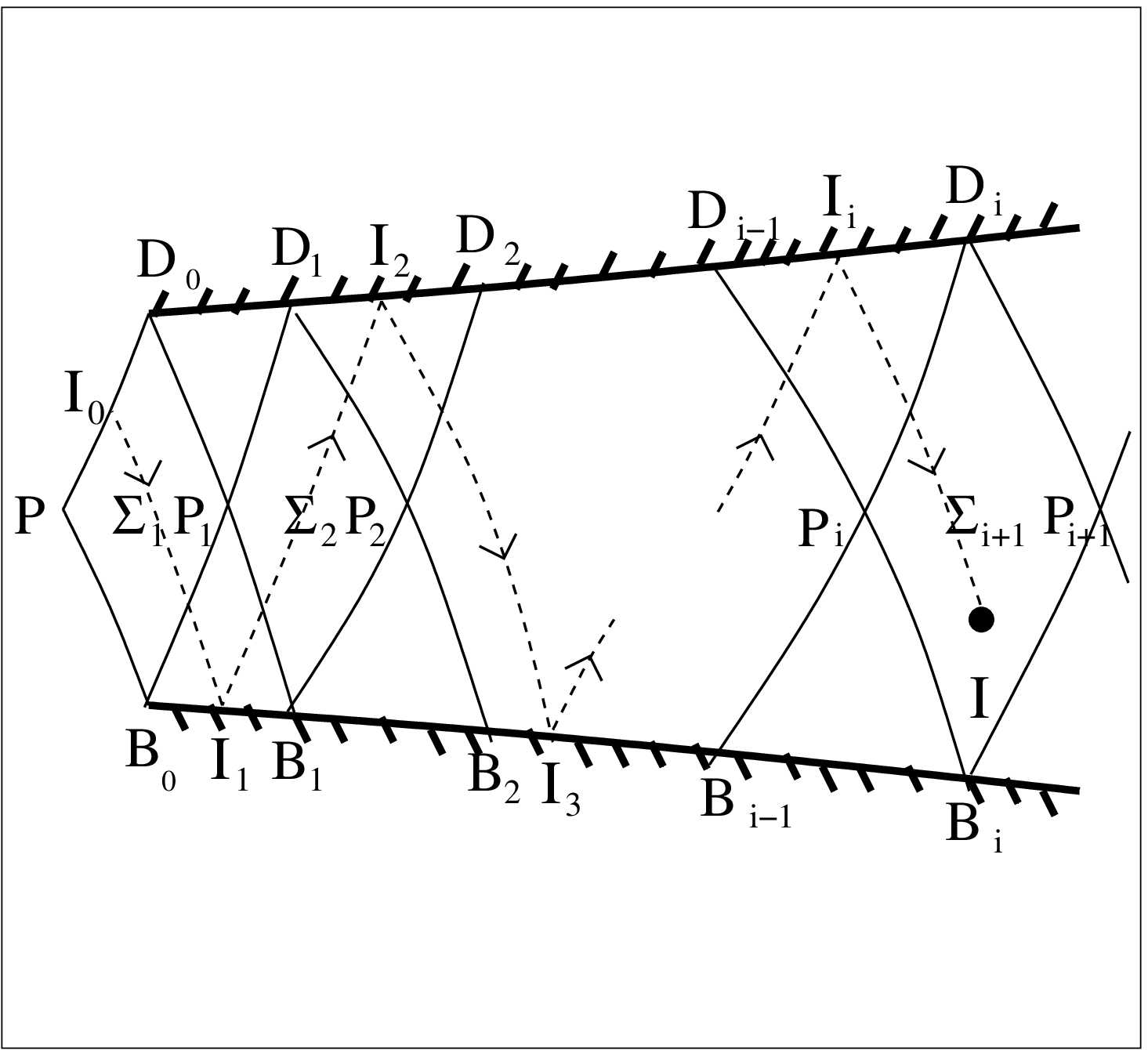}\quad\includegraphics[scale=0.29]{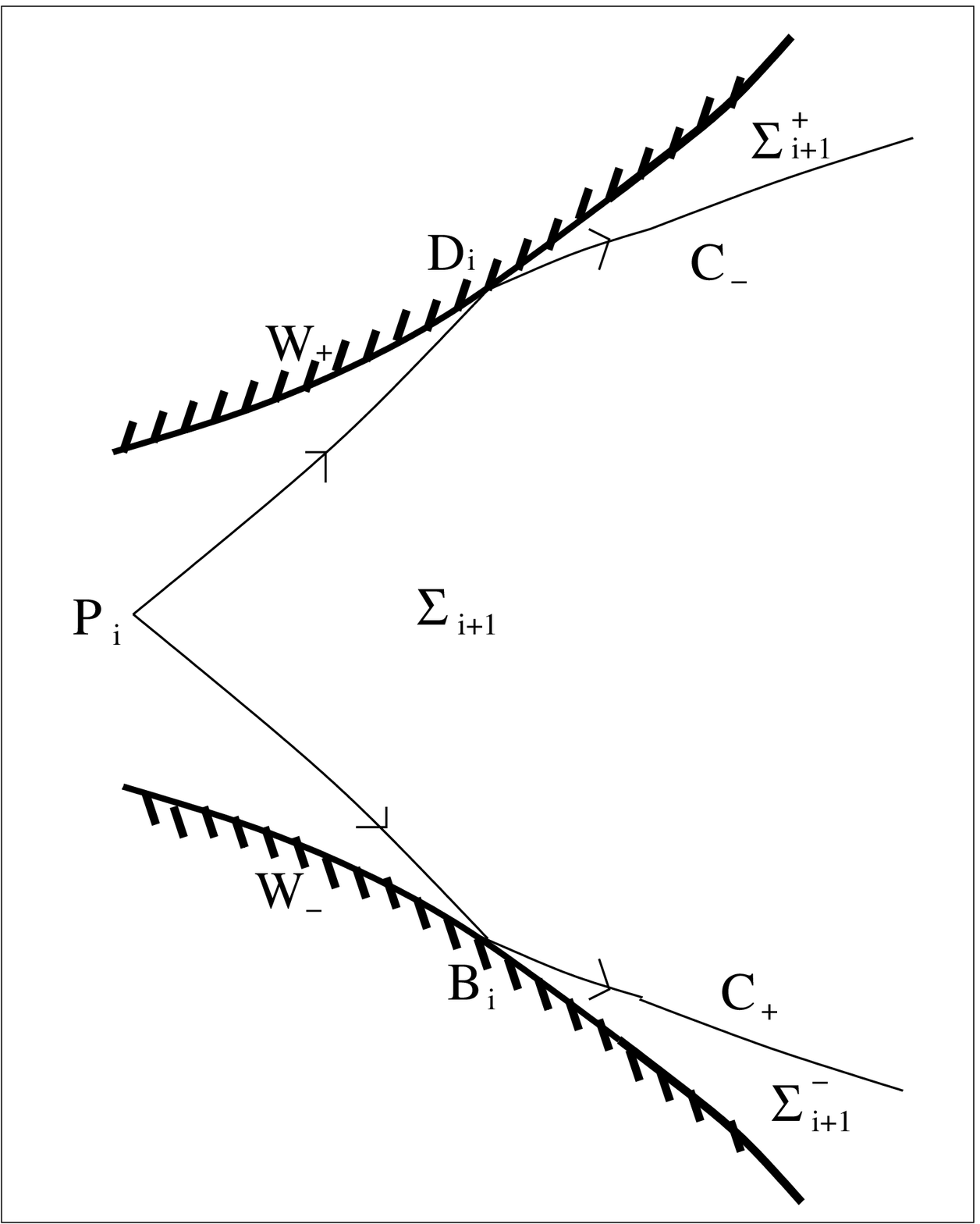}\quad \includegraphics[scale=0.365]{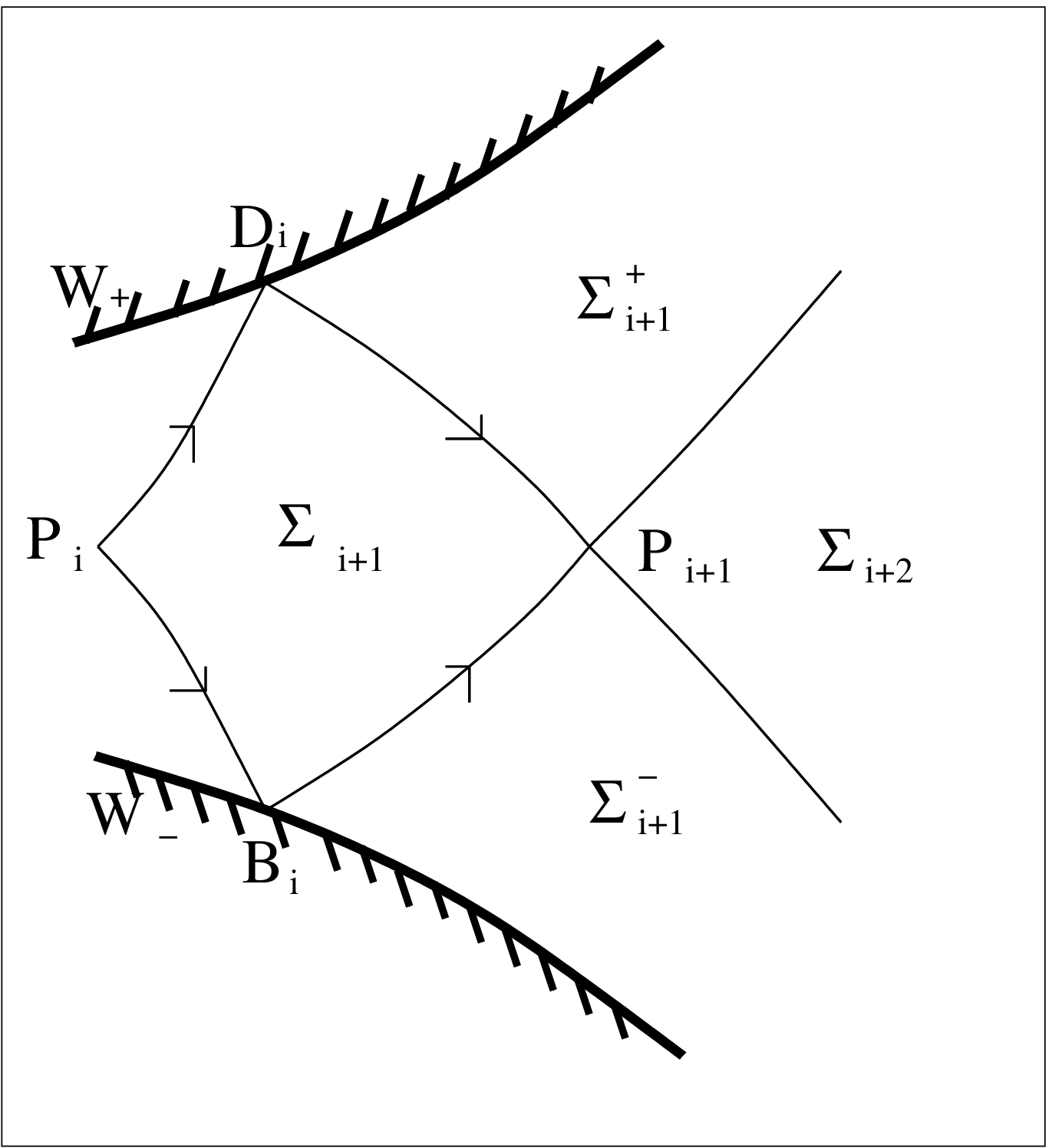}
\caption{ \footnotesize $C_{\pm}$ characteristic curves in the duct.}
\label{Fig6}
\end{center}
\end{figure}
 Suppose that
there are infinite  Goursat regions $\Sigma_i$ ($i=1, 2, 3, \cdot\cdot\cdot$).
Then, for each $i\geq 1$, $\Sigma_i$ is bounded by characteristic curves $\widehat{P_{i-1}B_{i-1}}$, $\widehat{P_{i-1}D_{i-1}}$, $\widehat{B_{i-1}P_{i}}$, and $\widehat{D_{i-1}P_{i}}$ as indicated in Figure \ref{Fig6} (left), where $P_0=P$.
Meanwhile, for any point $I\in \Sigma_{i+1}$, there is a $I_0\in \widehat{PD_0}\cup\widehat{PB_0}$, such that the forward $C_{+}$ or $C_{-}$ characteristic curve issuing from $I_0$ can reach $I$ after $i$ reflections on the walls.
Thus, by $\frac{\bar{\partial}_{\pm}c}{c^{\frac{2\gamma}{\gamma-1}}}<-\frac{m}{2}$ we know that there exists a sufficiently large $i\geq 0$, such that
$
c<c_{g}$ in $\Sigma_{i+1}$.
Consequently, we have that when  $\varepsilon<\varepsilon_0$ there holds
\begin{equation}\label{41401}
\bar{\partial}_{+}\alpha>0\quad \mbox{in}
\quad \Sigma_{i+1},
\end{equation}
as shown in (\ref{41301}).
Meanwhile, from (\ref{192307}), (\ref{41302}), and (\ref{198301}) we also have that when  $\varepsilon<\varepsilon_0$ there holds
\begin{equation}\label{41402}
\begin{aligned}
\bar{\partial}_{+}\alpha&~<~c^{\frac{2\gamma}{\gamma-1}-1}\tan A\left\{-(\kappa-1)\cos^2 A\frac{\bar{\partial}_{+}c }{c^{\frac{2\gamma}{\gamma-1}}} -\frac{\delta_1}{q^2(\gamma s)^{\frac{1}{\gamma-1}}}
+\frac{\delta_2\cos 2A}{\gamma(\gamma-1)s }\right\}\\&~<~\frac{c^{\frac{2\gamma}{\gamma-1}}}{q\cos A}\left\{-(\kappa-1)\cos^2 A\frac{\bar{\partial}_{+}c }{c^{\frac{2\gamma}{\gamma-1}}} +\frac{\widehat{\mathcal{M}}\varepsilon}{2}\right\}\\&~<~\frac{c^{\frac{2\gamma}{\gamma-1}}}{q\cos A}\left\{-(\kappa-1)\cos^2 A\frac{\bar{\partial}_{+}c }{c^{\frac{2\gamma}{\gamma-1}}} +\frac{(\kappa-1)m}{4}\right\}
\\&~<~\frac{c^{\frac{2\gamma}{\gamma-1}}}{q\cos A}\left\{-(\kappa-1)\cos^2 A\frac{\bar{\partial}_{+}c }{c^{\frac{2\gamma}{\gamma-1}}} -\frac{(\kappa-1)}{2}\frac{\bar{\partial}_{+}c }{c^{\frac{2\gamma}{\gamma-1}}}\right\}\\&
~<~-\frac{3(\kappa-1)}{2q\cos A}\bar{\partial}_{+}c ~<~-\frac{3(\kappa-1)}{q_m}\bar{\partial}_{+}c
\qquad \mbox{in}
\quad \Sigma_{i+1}.
\end{aligned}
\end{equation}
 Hence, along the forward $C_{+}$ characteristic curve passing through $B_{i}$ we have
$$
\begin{aligned}
\alpha&~<~\alpha(B_{i})+\frac{3(\kappa-1)}{q_m}\big(c(B_{i})-c\big)~<~\alpha(B_{i})+\frac{\arctan f'(x_{_{B_0}})}{3}\\&\qquad\quad~=~
-\arctan f'(x_{_{B_{i}}})+A(B_{i})+\frac{\arctan f'(x_{_{B_0}})}{3}~<~-\frac{\arctan f'(x_{_{B_0}})}{3}~<~0.
\end{aligned}
$$
Therefore, by symmetry we know that the forward $C_{+}$ characteristic curve issuing from $B_{i}$ and the forward $C_{-}$ characteristic curve issuing from $D_{i}$ do not intersect with each other; see Figure \ref{Fig6}(mid).
This implies that $\Sigma_{i+2}$ does not exist. This leads to a contradiction.




Therefore,
there are only the following two cases:
\begin{itemize}
  \item There exists an $i\geq 0$, such that the forward $C_{+}$ characteristic curve issuing from $B_{i}$ does not intersect with the forward $C_{-}$ characteristic curve issuing from $D_{i}$, as indicated Figure \ref{Fig6}(mid).
In this case, $$\Sigma=\Big(\bigcup\limits_{j=0}^{i+1}\Sigma_{j}\Big)\cup\Big(\bigcup\limits_{j=0}^{i+1}\Sigma_{j}^{+}\Big)
\cup\Big(\bigcup\limits_{j=0}^{i+1}\Sigma_{j}^{-}\Big).$$
  \item There exists an $i\geq 0$, such that the forward $C_{+}$ characteristic curve through $P_{i+1}$ does not intersect with $W_{+}$, and the forward $C_{-}$ characteristic curve through $P_{i+1}$ does not intersect with $W_{-}$ as indicated in Figure \ref{Fig6}(right).
In this case, $$\Sigma=\Big(\bigcup\limits_{j=0}^{i+2}\Sigma_{j}\Big)\cup\Big(\bigcup\limits_{j=0}^{i+1}\Sigma_{j}^{+}\Big)
\cup\Big(\bigcup\limits_{j=0}^{i+1}\Sigma_{j}^{-}\Big).$$
\end{itemize}
Therefore, we obtain a global piecewise smooth solution in the duct by solving a finite number of Gourst problems and slip boundary problems.

\subsection{Vacuum regions adjacent to the walls}
The purpose of this subsection is to discuss the appearance of vacuum.

From (\ref{42201}), we have
\begin{equation}
\begin{aligned}
\bar{\partial}_{0}c~=~&
\frac{1}{2\cos A}\big(\bar{\partial}_{-}c+\bar{\partial}_{+}c\big)
\\~=~&\frac{1}{2\cos A}\big(\bar{\partial}_{-}c-\bar{\partial}_{+}c+2\bar{\partial}_{+}c\big)\\~=~& \frac{1}{2\cos A}\Big(\frac{2q \bar{\partial}_{0}\sigma}{\kappa}-\frac{2j\bar{\partial}_{+}s}{\kappa}+2\bar{\partial}_{+}c\Big)
\\~=~& \frac{1}{2\cos A}\Big(\frac{2q \bar{\partial}_{0}\sigma}{\kappa}+2R_{+}\Big)
~<~\frac{q}{\kappa\cos A}\bar{\partial}_{0}\sigma~\quad \mbox{along}\quad W_{-}.
\end{aligned}
\end{equation}
Since $\bar{\partial}_{0}\widehat{E}=0$ and $\bar{\partial}_{0}\sigma\leq 0$ along $W_{-}$, we have
\begin{equation}
\bar{\partial}_{0}c~<~\frac{u_{in}\big(f(0)\big)}{\kappa}\bar{\partial}_{0}\sigma~\quad \mbox{along}\quad W_{-}.
\end{equation}

So, by integration we know that if $$\arctan f_{\infty}'>\kappa \frac{ c_{in}\big(f(0)\big)}{u_{in}\big(f(0)\big)}$$ then vacuum will appear on the wall $W_{-}$. That is, there is a $x_{_V}>0$ such that $c(x_{_V}, -f(x_{_V}))=0$. 

\begin{figure}[htbp]
\begin{center}
\includegraphics[scale=0.34]{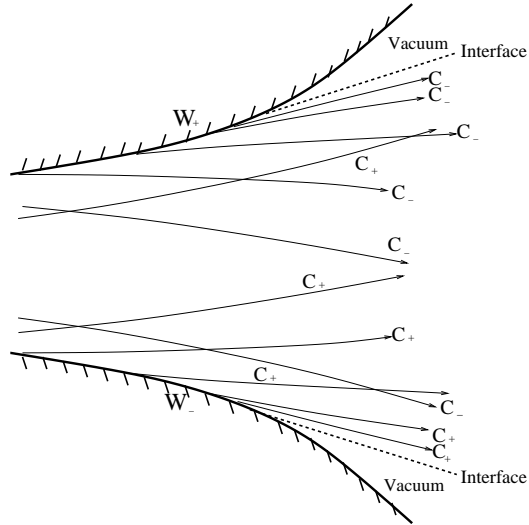}
\caption{ \footnotesize Vacuum regions adjacent to the walls.}
\label{Fig7}
\end{center}
\end{figure}

In what follows, we are going to show that if there is a vacuum then the vacuum is always adjacent to one of the walls, and the interface between gas and vacuum must be straight.

For $\tilde{x}<x_{_V}$, we denote by $y=y(x; \tilde{x})$, $x>\tilde{x}$ the $C_{+}$ characteristic curve issuing from the point $(\tilde{x}, -f(\tilde{x}))$. As shown in (\ref{41401}) and (\ref{41402}), we can prove that when $\tilde{x}$ is sufficiently close $x_{_V}$,
$$
c<c_{g}\quad\mbox{and}\quad 0<\bar{\partial}_{+}\alpha<-\frac{3(\kappa-1)}{q_m}\bar{\partial}_{+}c
\quad \mbox{along}\quad  y=y(x; \tilde{x}),~ x>\tilde{x}.
$$
Thus, we have
$$
\arctan y'(x; \tilde{x})-\arctan y'(\tilde{x}; \tilde{x})~<~ \frac{3(\kappa-1)}{q_m}c(\tilde{x}, -f(\tilde{x}))\quad \mbox{as}\quad x>\tilde{x}.
$$
In addition,
 since $c(\tilde{x}, -f(\tilde{x}))\rightarrow 0$
and $$\arctan y'(\tilde{x}; \tilde{x})=-\arctan f'(\tilde{x})+A(\tilde{x}, -f(\tilde{x}))\rightarrow -\arctan f'(x_{_V})$$
as $\tilde{x}\rightarrow x_{_V}$, we have that for any $X>x_{_V}$,
$$\lim\limits_{\tilde{x}\rightarrow x_{_V}} \Big\|y(x; \tilde{x})+f(x_{_V})+f'(x_{_V})(x-x_{_V})\Big\|_{C[x_{_V}, X]}=0.$$
Therefore, there are no gas flow into the region $\big\{(x,y)\mid -f(x)<y<-f(x_{_V})-f'(x_{_V})(x-x_{_V}), ~x>x_{_V}\big\}.$
By symmetry we also have that there are also no gas flow into the region $\big\{(x,y)\mid f(x_{_V})+f'(x_{_V})(x-x_{_V})<y<f(x), ~x>x_{_V}\big\}.$
See Figure \ref{Fig7}.

From Lemmas \ref{lem33} and \ref{lem36} we can see that $c>0$ in
$
\big\{(x, y)\mid -g(x)~<~y~<~g(x),~  x>0\big\},
$
where
$$
g(x)=\left\{
       \begin{array}{ll}
         f(x), & \hbox{$0<x<x_{_V}$;} \\[4pt]
         f(x_{_V})+f'(x_{_V})(x-x_{_V}), & \hbox{$x\geq x_{_V}$.}
       \end{array}
     \right.
$$

Since $\varepsilon\rightarrow 0$ as $\epsilon\rightarrow 0$,  we complete the proof of Theorem \ref{main}.




\vskip 32pt

\end{document}